\documentclass[twoside,12pt]{article}

\usepackage{amsmath}
\usepackage{amssymb}
\usepackage{pxfonts}
\usepackage{graphicx}
\usepackage{theorem}
\usepackage{pifont}
\usepackage{color}
\usepackage[all]{xy}
\usepackage{tikz}
\usepackage{tkz-euclide}

\usetkzobj{all}

\newcommand{\brk}[1]{\left(#1\right)}          
\newcommand{\Brk}[1]{\left[#1\right]}          
\newcommand{\BRK}[1]{\left\{#1\right\}}        
\newcommand{\beq}{\begin{equation}}
\newcommand{\eeq}{\end{equation}}
\newcommand{\bbR}{{\mathbb R}}
\providecommand{\R}{\bbR}

\newcommand{\bbN}{{\mathbb N}}
\newcommand{\calM}{{\mathcal M}}
\newcommand{\calA}{\mathcal{A}}

\newcommand{\textand}{\quad\text{ and }\quad}
\newcommand{\Textand}{\qquad\text{ and }\qquad}

\newcommand{\defref}[1]{Definition~\ref{#1}}

\newcommand{\lemref}[1]{Lemma~\ref{#1}}

\theoremheaderfont{\fontfamily{pzc}\bfseries\large}
\newtheorem{theorem}{Theorem}[section]
\newtheorem{lemma}[theorem]{Lemma}
\newtheorem{proposition}[theorem]{Proposition}

\newtheorem{corollary}[theorem]{Corollary}
\newtheorem{definition}[theorem]{Definition}

{\theorembodyfont{\rmfamily}
\newenvironment{proof}{{\flushleft \emph{Proof}:}}{\hfill\ding{110}}

\newenvironment{remark}{{\flushleft \fontfamily{pzc}\bfseries\large Remark:}}{}

\newcommand{\g}{\mathfrak{g}}
\newcommand{\h}{\mathfrak{h}}
\newcommand{\euc}{\mathfrak{e}}
\newcommand{\Vol}{d\text{Vol}}
\newcommand{\Volume}{d\text{Vol}_\g}

\newcommand{\Volumeh}{d\text{Vol}_{\mathfrak{h}}}
\newcommand{\textVol}{\text{Vol}}

\newcommand{\M}{{\calM}}
\newcommand{\N}{\mathcal{N}}

\newcommand{\id}{{\text{Id}}}

\newcommand{\vp}{\varphi}

\newcommand{\dist}{\operatorname{dist}}

\newcommand{\SO}[1]{\text{SO}(#1)}

\newcommand{\limn}{\lim_{n\to\infty}}

\newcommand{\tR}{\tilde{R}}
\newcommand{\tM}{\tilde{\M}}
\newcommand{\tH}{\tilde{H}}
\newcommand{\hH}{\hat{H}}
\newcommand{\hK}{\hat{K}}

\newcommand{\e}{\varepsilon}

\newcommand{\pl}{\partial}

\newcommand{\dis}{\operatorname{dis}}

\newcommand{\Raz}[1]{{\color{blue} #1}}
\newcommand{\Cy}[1]{{\color{red} #1}}

\renewcommand{\Cy}[1]{{ #1}}

\numberwithin{equation}{section}

\begin{document}

\title{The emergence of torsion in the continuum limit of distributed \Cy{edge-}dislocations}
\author{Raz Kupferman and Cy Maor}
\date{}
\maketitle

\begin{abstract}
We present a rigorous homogenization theorem for distributed \Cy{edge-}dislocations.
We construct a sequence of locally-flat \Cy{2D} Riemannian manifolds with dislocation-type singularities. We show that this sequence converges, as the dislocations become denser, to a flat non-singular Weitzenb\"ock manifold, i.e. a flat manifold endowed with a metrically-consistent connection with zero curvature and non-zero torsion.
In the process, we introduce a new notion of convergence of Weitzenb\"ock manifolds, which is relevant to this class of homogenization problems.
\end{abstract}

\tableofcontents

\section{Introduction}

\paragraph{Manifolds with dislocations}
The study of defects in solids with imperfections is a longstanding theme in material science.
There exists a wide range of prototypical crystalline defects, among which are dislocations, disclinations and point defects (see Kr\"oner \cite{Kro81} for a classical review).
The common practice in crystallography is to identify and quantify defects of dislocation-type  via \emph{Burgers circuits}, 
which are based on discrete steps with respect a local crystalline structure. Defects are quantified by  the \emph{Burgers vector}, which is a discrepancy between closed loop in real space and closed loops in the discrete ``ideal" crystallographic space.

Dislocations have also been considered in the context of amorphous materials. More than a century ago, Volterra constructed a variety of defects using ``cut-and-weld" procedures \cite{Vol07}.  A Burgers vector arises naturally in this context too, with the crystallographic structure replaced by the Riemannian metric and its associated parallel transport \cite{OY14}.
Recently, Kupferman et al. \cite{KMS14} introduced a general approach to describe isolated defects in amorphous materials, using the differential geometric notion of monodromy in affine manifolds. The Burgers vector is identified with the translational component of the monodromy, whereas its rotational component quantifies the magnitude of disclination-type defects (the Frank vector).

In the above-mentioned approaches to isolated defects, 
the continuum is modeled as a topological manifold, smooth everywhere except at the loci of the defects. The smooth part of the manifold is endowed with a locally-flat Riemannian metric. The defects, which are the singularities of the topological manifold, manifest  through the properties of the Riemannian (Levi-Civita) parallel transport. The important observation is that when considering isolated defects in amorphous materials, the Riemannian structure is the only structure imposed on the material manifold.

\paragraph{Continuously distributed dislocations}
It is customary in material science to consider materials with distributed defects. In the spirit of continuum mechanics, bodies with distributed dislocations were modeled as smooth manifolds, starting in the 1950s with the pioneering work of Nye \cite{Nye53}, \Cy{Kondo \cite{Kon55}} and Bilby et al. \cite{BBS55,BS56}. In these works, the singularities were smoothed out, resulting in a manifold endowed with a flat metric, and in addition, a \emph{torsion field} that represents the Burgers vector density. In other words, the presence of distributed dislocations was modeled by an additional geometric structure imposed on the material manifold. 

This classical modelling of distributed dislocations is  phenomenological.
A natural question is in what sense does torsion emerge in the continuum limit of discretely distributed dislocations. That is, one would hope to obtain torsion as a homogenization limit of an increasingly large number of discrete dislocations.

\paragraph{Outline of results}

In this paper we construct a sequence of manifolds with isolated dislocations, such that the dislocations become increasingly dense, while their total magnitude remains fixed.
We show the convergence of both metric and parallelism. (i) The sequence converges as a sequence of metric spaces to a flat, simply-connected Riemannian manifold. (ii) The sequence converges as a sequence of manifolds with connections. The Levi-Civita connections converge in a weak sense to a  metrically-consistent \emph{non-symmetric} connection. This means that a torsion field arises in a rigorous limit process from torsion-free Riemannian manifolds. 
 This notion of convergence of Weitzenb\"ock manifolds  with connections is, to our knowledge,  new.

\paragraph{Structure of this paper}
In Section \ref{sec:sequence} we describe the construction of a manifold with a single edge-dislocation, and then construct a sequence of manifolds with increasingly dense dislocations. 
In Section \ref{sec:conv_example} we prove that this sequence of manifolds converges to a Weitzenb\"ock manifold (a Riemannian manifold endowed with a metrically-consistent, flat, non-symmetric connection). 
This example leads us in Section \ref{sec:conv} to a definition of convergence of Weitzenb\"ock manifolds. We prove that this notion of convergence is well-defined.
Finally, we discuss in Section \ref{sec:discussion} the properties of the limit manifold, and relate   the limit connection to  Burgers vectors and dislocation line densities.

\paragraph{A note about mechanics and geometry} 
Torsion appears also in a mechanical context, where it is related to symmetries in the constitutive laws (Wang \cite{Wan67}).  The present paper does not consider the mechanical implications of defects. 
The homogenization process described in the paper can be posed, 
in essence, in pure geometric terms. 

\section{A sequence of locally-flat manifolds  with  defects}
\label{sec:sequence}

\subsection{A single edge-dislocation}

Consider the Euclidean plane that undergoes the following Volterra cut-and-weld procedure  \cite{Vol07}:
First, remove a sector of angle $2\theta<\pi$, and glue together (i.e., identify) the two rays that were the boundaries of the sector. This results in a locally-Euclidean surface with a cone singularity at a point which we denote by $p_+$.
Next, choose a point $p_-$ at a distance $d$ from $p_+$, and cut the surface along a ray that starts at $p_-$ and does not pass through $p_+$. Finally, insert into the cut a sector of angle $2\theta$, with its vertex at $p_-$ and its two sides glued to the edges of the cut (see Figure~\ref{fig:volterra}).

\begin{figure}
\begin{center}
\begin{tikzpicture}
	\fill[color=gray!5] (-3,-3.2) -- (7.5,-3.2) -- (7.5,3.2) -- (-3,3.2) -- cycle;
	\tkzDefPoint(0,0){O}
	\tkzDefPoint(-3,0.5){A}
	\tkzDefPoint(-3,-0.5){B}
	\tkzDrawSegment(O,A);
	\tkzDrawSegment(O,B)
	\tkzDrawArc[dotted](O,B)(A)
	\tkzMarkSegment[mark=s||](O,A)
	\tkzMarkSegment[mark=s||](O,B)
	\tkzDrawPoint(O)
	\tkzLabelPoint[above](O){$p_+$}
	\tkzDefPoint(1,0){pm}
	\tkzDrawPoint(pm)
	\tkzLabelPoint[above](pm){$p_-$}
	\tkzDefPoint(5,0){dummy}
	\tkzDrawSegment[dashed](pm,dummy)
	\tkzText(1.7,0){\ding{36}}
	\tkzText(2.5,0){\ding{36}}
	\tkzText(3.3,0){\ding{36}}
	\tkzText(4.1,0){\ding{36}}
	
	\tkzDefPoint(4.5,0){O1}
	\tkzDrawPoint(O1)

	\tkzDefPoint(7.5,0.5){A1}
	\tkzDefPoint(7.5,-0.5){B1}
	\tkzDrawSegment(O1,A1);
	\tkzDrawSegment(O1,B1)
	\tkzDrawArc[dotted](O1,B1)(A1)
	\draw[->] (6.0,0.75) -- (4.0,0.5);
	\draw[->] (6.0,-0.75) -- (4.0,-0.5);
	\tkzLabelPoint[above](O1){$p_-$}

\end{tikzpicture}
\end{center}
\caption{The Volterra cut-and-weld construction of a curvature dipole, or an edge-dislocation. A sector whose vertex is denoted by $p_+$ is removed from the plane and its outer boundaries are glued together, thus forming a cone. The same sector is then inserted into a straight cut along a ray whose endpoint is denoted by $p_-$.}
\label{fig:volterra}
\end{figure}
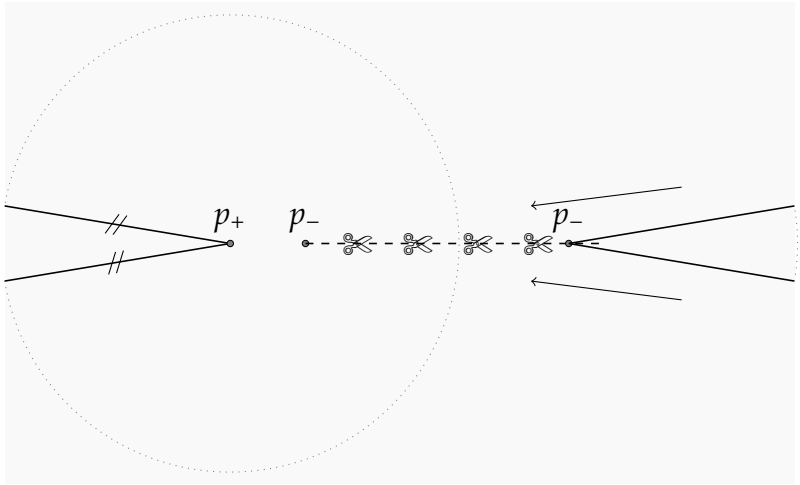

In material-science terminology, we obtain a plane with a pair of disclinations of equal magnitudes and opposite signs.  This pair of disclinations is the isotropic equivalent of a pentagon-heptagon pair in an hexagonal lattice, which is another realization of an edge-dislocation (Seung and Nelson \cite{SN88}, \Cy{see comment in Section~2.2}).

Mathematically, we obtain a simply-connected topological manifold which carries a structure of a complete metric space. The points $p_\pm$ are said to carry a cone singularity: $p_+$ is the vertex of a cone and $p_-$ is the vertex of an anti-cone.
Removing the singular points $p_\pm$, we obtain a locally-flat (or locally-Euclidean) Riemannian manifold. This means that every point has a neighborhood that is isometric to an open subset of the Euclidean plane. It we further remove the segment that connects $p_+$ and $p_-$ we obtain a Riemannian manifold with trivial holonomy: parallel transport is path-independent.  As a result, this manifold can be covered by a parallel frame field. 

A few comments:
(i) The distance between $p_+$ and $p_-$ after the cut-and-weld procedure is still $d$, and the shortest path between those points is the same as in the original plane, hence the segment between $p_+$ and $p_-$ is well-defined. We call this segment the \emph{dislocation line}. (ii) Parallel transport and holonomy are with respect to the Levi-Civita connection.
(iii) This construction was studied in detail by Guven et al. \cite{GHKM13}; see also Section 4.3 in \cite{KMS14}.

\subsection{The building-block $R(a,b,\theta,\e)$}
\label{subsec:R}

We next consider a compact subset of a plane with a single edge-dislocation. In Subsection~\ref{subsec:multiple} it is used as a building block for surfaces with multiple edge-dislocations. 

From the vertex $p_-$ of the anti-cone emanate two rays that are at an angle of $\pi/2+\theta$  from the dislocation line $[p_-,p_+]$. These two rays partition the surface into two sets. Since the total angle around the anti-cone is $2\pi+ 2\theta$, the set that does not include $p_+$ is a (non-singular) half-plane, which we denote by $X_-$. Likewise, we denote by $X_+$ the half-plane whose boundary intersects the dislocation line $[p_-,p_+]$ at $p_+$ at an angle $\pi/2-\theta$
(see Figure~\ref{fig:R}). 

\begin{figure}
\begin{center}
\begin{tikzpicture}
	\fill[opacity=0.1] (2,4)--(-1,4)--(-1,-1)--(2,-1)--cycle;
	\fill[opacity=0.1] (4,4)--(7,4)--(7,-1)--(4,-1)--cycle;
	\tkzDefPoint(0,3){A}
	\tkzDefPoint(0,0){B}
	\tkzDefPoint(6,0){C}
	\tkzDefPoint(6,3){D}
	\tkzDefPoint(4,1.5){pp}
	\tkzDefPoint(2,1.5){pm}
	\tkzDefPoint(4,-1){pp1}
	\tkzDefPoint(4,4){pp2}
	\tkzDefPoint(2,-1){pm1}
	\tkzDefPoint(2,4){pm2}
	\tkzDefPoint(0.7,3.8){lab1}
	\tkzDefPoint(5.3,3.8){lab2}
	\tkzDefPoint(2.52,0.75){lab3}
	\tkzDefPoint(3.5,0.75){lab4}
	\tkzDrawPoints(A,B,C,D,pp,pm);
	\tkzLabelPoint[left](A){$A$};
	\tkzLabelPoint[left](B){$B$};
	\tkzLabelPoint[right](C){$C$};
	\tkzLabelPoint[right](D){$D$};
	\tkzLabelPoint[right](pp){$p_+$};
	\tkzLabelPoint[left](pm){$p_-$};
	\tkzDrawSegment(A,B);
	\tkzDrawSegment(A,D);
	\tkzDrawSegment(C,B);
	\tkzDrawSegment(C,D);
	\tkzDrawSegment[dashed](pp,pm);
	\tkzLabelSegment(pp,pm){$d$};
	\tkzLabelSegment[left](A,B){$a$};
	\tkzLabelSegment[above=0.15cm](A,D){$b$};
	\tkzLabelSegment[below=0.15](B,C){$b$};
	\tkzLabelSegment[right](C,D){$a+\e$};
	\tkzDrawSegment[dashed](pp1,pp2);
	\tkzDrawSegment[dashed](pm1,pm2);
	\tkzMarkAngle[size=0.6](pm1,pm,pp);
	\tkzMarkAngle[size=0.6](pm,pp,pp1);
	\tkzText[rotate=-60](lab3){{\footnotesize $\pi/2+\theta$}};
	\tkzText[rotate=60](lab4){{\footnotesize $\pi/2-\theta$}};
	\tkzText(lab1){$X_-$}
	\tkzText(lab2){$X_+$}
	\draw[<->,dotted] (0,3.2) -- (6,3.2);
	\draw[<->,dotted] (0,-0.2) -- (6,-0.2);
\end{tikzpicture}
\end{center}
\caption{The building block $R(a,b,\theta,\e)$.}
\label{fig:R}
\end{figure}
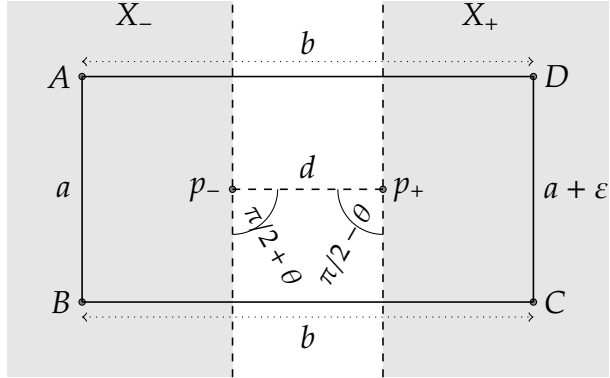

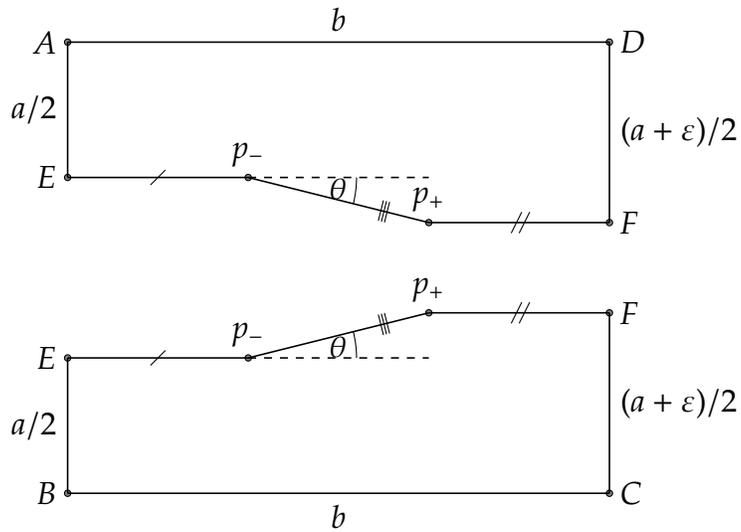
\begin{figure}
\begin{center}
\begin{tikzpicture}[scale=1.2]
	\tkzDefPoint(0,4){A}
	\tkzDefPoint(0,2.5){A1}
	\tkzDefPoint(2,2.5){pm}
	\tkzDefPoint(4,2.5){dummy}
	\tkzDefPoint(4,2){pp}
	\tkzDefPoint(6,2){D1}
	\tkzDefPoint(6,4){D}
	\tkzDrawPoints(A,A1,pm,pp,D1,D);
	\tkzLabelPoint[left](A){$A$};
	\tkzLabelPoint[left](A1){$E$};
	\tkzLabelPoint[right](D){$D$};
	\tkzLabelPoint[right](D1){$F$};
	\tkzLabelPoint[above](pm){$p_-$};
	\tkzLabelPoint[above](pp){$p_+$};
	\tkzDrawSegment(A,D);
	\tkzDrawSegment(D,D1);
	\tkzDrawSegment(A,A1);
	\tkzDrawSegment(A1,pm);
	\tkzDrawSegment(pp,pm);
	\tkzDrawSegment(pp,D1);
	\tkzDrawSegment[dashed](pm,dummy);
	\tkzLabelAngle(dummy,pm,pp){{\small $\theta$}};
	\tkzMarkAngle[size=1.2](pp,pm,dummy);
	\tkzLabelSegment(A,D){$b$};
	\tkzLabelSegment[left](A,A1){$a/2$};
	\tkzLabelSegment[right](D,D1){$(a+\e)/2$};
	\tkzMarkSegment[mark=s|](A1,pm)
	\tkzMarkSegment[pos=0.75, mark=|||](pm,pp)
	\tkzMarkSegment[mark=s||](pp,D1)

	\tkzDefPoint(0,-1){B}
	\tkzDefPoint(0,0.5){B1}
	\tkzDefPoint(2,0.5){pm}
	\tkzDefPoint(4,0.5){dummy}
	\tkzDefPoint(4,1){pp}
	\tkzDefPoint(6,1){C1}
	\tkzDefPoint(6,-1){C}
	\tkzDrawPoints(B,B1,pm,pp,C1,C);
	\tkzLabelPoint[left](B){$B$};
	\tkzLabelPoint[left](B1){$E$};
	\tkzLabelPoint[right](C){$C$};
	\tkzLabelPoint[right](C1){$F$};
	\tkzLabelPoint[above](pm){$p_-$};
	\tkzLabelPoint[above](pp){$p_+$};
	\tkzDrawSegment(B,C);
	\tkzDrawSegment(C,C1);
	\tkzDrawSegment(B,B1);
	\tkzDrawSegment(B1,pm);
	\tkzDrawSegment(pp,pm);
	\tkzDrawSegment(pp,C1);
	\tkzDrawSegment[dashed](pm,dummy);
	\tkzLabelAngle(dummy,pm,pp){{\small $\theta$}};
	\tkzMarkAngle[size=1.2](dummy,pm,pp);
	\tkzLabelSegment[below](B,C){$b$};
	\tkzLabelSegment[left](B,B1){$a/2$};
	\tkzLabelSegment[right](C,C1){$(a+\e)/2$};
	\tkzMarkSegment[mark=s|](B1,pm)
	\tkzMarkSegment[pos=0.75, mark=|||](pm,pp)
	\tkzMarkSegment[mark=s||](pp,C1)
\end{tikzpicture}
\end{center}
\caption{An alternative construction of the building block $R(a,b,\theta,\e)$.}
\label{fig:R2}
\end{figure}

We  construct a ``rectangle" $ABCD$ as follows:
\begin{enumerate}
\item Choose a point $A\in X_-$. 
\item Let $B\in X_-$ be the unique point such that $AB$ is parallel to the boundary of $X_-$ and $d(A,p_-)=d(B,p_-)$. Denote $|AB| = a$.
\item Choose $C\in X_+$ such that $BC\perp AB$. Denote $|BC|=b$. 
\item Let $D\in X_+$ be the unique point such that $AD\perp AB$ and $|AD| = b$.
\end{enumerate}

\paragraph{Comments}\ 
\begin{enumerate}
\item
An alternative representation of the same ``rectangle" is displayed in  Figure~\ref{fig:R2}.
The figure shows two hexagons, $ADFp_+p_-E$ and $BCFp_+p_-E$ (both are bone-fide Euclidean hexagons). The ``rectangle" is formed by identifying the 
segments $Ep_-$, $p_-p_+$ and $p_+F$ in both hexagons. 
\Cy{
This representation shows how the combination of two disclinations is metrically equivalent to standard description of a two-dimensional  edge-dislocation, which is the insertion of a half-line.
}

\item It follows (for example from the alternative representation in Figure~\ref{fig:R2}) that $BC\perp CD$ and $AD\perp CD$, hence $ABCD$ can be thought of as a rectangle with singularities. 
Note that Figure~\ref{fig:R} is somewhat misleading as this ``rectangle" cannot be embedded in the plane.

\item 
It is easy to see from Figure~\ref{fig:R2} that
\[
|CF| = |DF| = \frac{a+\e}{2},
\]
where
\beq
\e = 2d\sin\theta.
\label{eq:eps}
\eeq
Looking back at Figure~\ref{fig:R} we have $|CD| = a + \e$, i.e., the ``rectangle" $ABCD$ does not satisfy the Euclidean property of having opposite sides of equal length. The parameter $\e$ is the excess in length of the longer side, and measures the magnitude of the dislocation.

 \end{enumerate}

The above ``rectangle" is a simply-connected topological manifold with boundary, which we denote by $\tR(a,b,\theta,\e)$. Note that the parameters $a,b,\theta,\e$ do not  determine the shape uniquely, as the position of the dislocation line $[p_-,p_+]$ can be shifted horizontally. In reference to  Figure~\ref{fig:R2},
\[
|Ep_-| + d\,\cos\theta + |p_+F| = b.
\]
Without loss of generality we assume $|Ep_-| = |p_+F|$, thus determining $\tR(a,b,\theta,\e)$ uniquely (we will see later that the exact position of the dislocation line does not affect the limit). 
We also denote 
\[
R(a,b,\theta,\e) = \tR(a,b,\theta,\e)\setminus [p_-,p_+],
\]
which is a non-compact smooth manifold with corners. The  Levi-Civita parallel transport in $R(a,b,\theta,\e)$ is  path-independent, which is the reason we remove the whole dislocation line $[p_-,p_+]$ rather than only the singular points $p_{\pm}$. 

\subsection{Manifolds with multiple edge-dislocations}
\label{subsec:multiple}

By using $\tR(a,b,\theta,\e)$ as a building block and gluing copies together, we  generate manifolds with multiple edge-dislocations. Since our goal is to investigate a limit process in which dislocations get denser, we need to ``zoom out", or in other words, rescale the space in an appropriate way. We do so by
constructing manifolds that have a fixed boundary, a fixed total dislocation magnitude $\e$, and are partitioned into an increasing number of $\tR$-blocks.

Fix $a$, $b$, $\theta$ and  $\e$.
Given $n\in\bbN$, we construct a topological manifold with corners $\tM_n$ by gluing together $n^2$ building blocks, where the $(i,j)$-th block, which we denote by $\tM_n(i,j)$ is of  type 
\beq
\tM_n(i,j) = \tR\brk{a_{n,i},b_n,\theta,\e_n},
\label{eq:Rk}
\eeq
where
\beq
a_{n,i} = \frac{a+(i-1)\e/n}{n},
\qquad
b_n = \frac{b}{n}
\textand
\e_n = \frac{\e}{n^2}
\label{eq:abeps}
\eeq
(see Figure~\ref{fig:Mk}).

\begin{figure}
\begin{center}
\begin{tikzpicture}[scale=0.8]
	\tkzDefPoint(0,0){A}
	\tkzDefPoint(15,0){B}
	\tkzDefPoint(15,5){C}
	\tkzDefPoint(0,5){D}
	\tkzDrawPolygon(A,B,C,D)
	\tkzLabelSegment[below](A,B){$b$}
	\tkzLabelSegment[above](C,D){$b$}
	\tkzLabelSegment[left](A,D){$a$}
	\tkzLabelSegment[right](B,C){$a+\e$}
	\draw[dashed] (3,0) -- (3,5);
	\draw[dashed] (6,0) -- (6,5);
	\draw[dashed] (9,0) -- (9,5);
	\draw[dashed] (12,0) -- (12,5);
	\draw[dashed] (0,1.25) -- (15,1.25);
	\draw[dashed] (0,2.5) -- (15,2.5);
	\draw[dashed] (0,3.75) -- (15,3.75);
	\tkzText(1.5,0.5){{\footnotesize $\tM_n(1,1)$}}
	\tkzText(4.5,0.5){{\footnotesize $\tM_n(2,1)$}}
	\tkzText(7.5,0.5){{\footnotesize $\tM_n(3,1)$}}
	\tkzText(10.5,0.5){$\cdots$}
	\tkzText(13.5,0.5){{\footnotesize $\tM_n(n,1)$}}
	\tkzText(1.5,4.25){{\footnotesize $\tM_n(1,n)$}}
	\tkzText(4.5,4.25){{\footnotesize $\tM_n(2,n)$}}
	\tkzText(7.5,4.25){{\footnotesize $\tM_n(3,n)$}}
	\tkzText(10.5,4.25){$\cdots$}
	\tkzText(13.5,4.25){{\footnotesize $\tM_n(n,n)$}}
	\tkzText(1.5,1.75){$\vdots$}
	\tkzText(1.5,3){$\vdots$}
	\tkzText(13.5,1.75){$\vdots$}
	\tkzText(13.5,3){$\vdots$}
\end{tikzpicture}
\end{center}
\caption{The manifold $\tM_n$ obtained by gluing together $n^2$ building blocks. At the $i$-th column and $j$-th row we place the ``rectangle" $\tM_n(i,j)$ defined by \eqref{eq:Rk} and \eqref{eq:abeps}.}
\label{fig:Mk}
\end{figure}
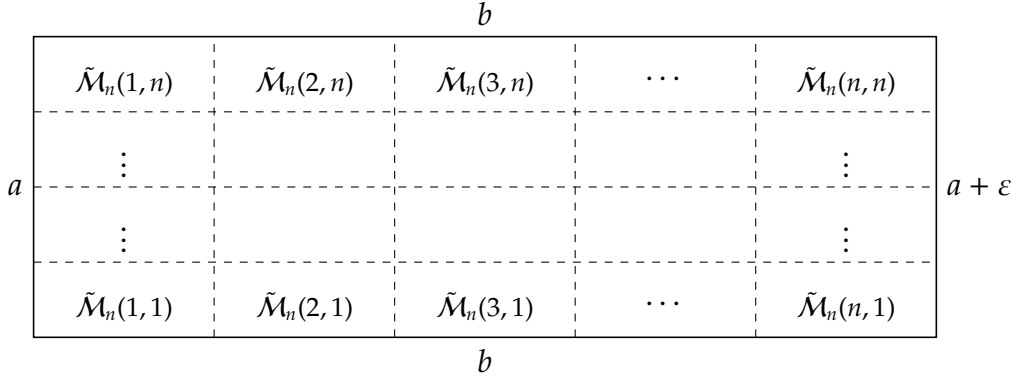

The rectangular nature of the  $\tM_n(i,j)$-blocks enables us to glue them such that the manifold is smooth 
across the  blocks. The only singularities in $\tM_n$ are the points $p_{\pm}$ in each $\tM_n(i,j)$.
The singularities do not get milder as $n$ increases, since $\theta$ remains fixed. The distance between pairs of singular points $p_\pm$ decreases, however, by \eqref{eq:eps}  as $1/n^2$. If we describe the defects as curvature multipoles, the monopoles remain constant but the dipoles decrease like $1/n^2$.

We  denote by $\M_n$ the amalgam of $n^2$ $R$-blocks.   
The manifolds $\M_n$ form a sequence of smooth manifolds with corners,  satisfying the following properties:

\begin{enumerate}
\item They are locally Euclidean; we denote the Riemannian metrics by $\g_n$.
\item The boundary is $n$-independent; $\partial\M_n$ consists of four segments of length $a,b,b$ and $a+\e$ that form a ``rectangle".
\item The  parallel transport operator $\Pi_n$, induced by the Levi-Civita connection $\nabla_n$ is path-independent since it is inherited from the parallel transport within the $R$-blocks.
\item The completion of $\M_n$ as a metric space is the simply-connected topological manifold with corners $\tM_n$. We denote by $d_n$ the distance function in $\tM_n$.
\end{enumerate}

\section{Convergence to a non-singular manifold with torsion}
\label{sec:conv_example}

In this section we show that the sequence $(\M_n,\g_n,\nabla_n)$ converges to a compact, non-singular,   simply-connected Riemannian manifold with corners $(\N,\g)$, endowed with a metrically-compatible \emph{non-symmetric} connection $\nabla$. 

We start by defining $(\N,\g,\nabla)$.
Denote by $\N=\N(a,b,\e)$ the sector of angle $\e/b$ of an annulus of inner radius $R_0 = ab/\e$ and outer radius $R_1= R_0+b =(a+\e)b/\e$. Endow $\N$ with the standard Euclidean metric, denoted by $\g$; the corresponding distance function is denoted by $d$. $(\N,\g)$ is a manifold with corners whose edges have lengths $a$, $b$, $b$ and $a+\e$
(see Figure~\ref{fig:M}).

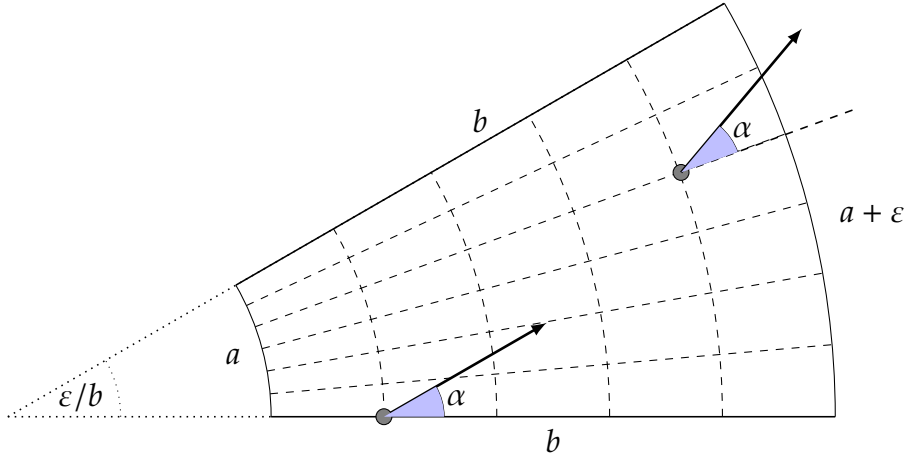
\begin{figure}
\begin{center}
\begin{tikzpicture}
	\tkzDefPoint(0,0){O}
	\tkzDefPoint(11,0){A}
	\tkzDefPoint(3.03108891324553526366,1.75){B}
	\tkzDefPoint(3.5,0){A1}
	\tkzDefPoint(9.52627944162882511436,5.5){B1}
	\tkzDrawSegment(A1,A);
	\tkzDrawSegment(B1,B);
	\tkzDrawSegment[dotted](O,A1);
	\tkzDrawSegment[dotted](O,B1);
	\tkzLabelAngle(A1,O,B1){$\e/b$}
	\tkzMarkAngle[size=1.5,dotted](A1,O,B1)
	\draw (11,0) arc (0:30:11);
	\draw (3.5,0) arc (0:30:3.5);
	\draw[dashed] (5.,0) arc (0:30:5.);
	\draw[dashed] (6.5,0) arc (0:30:6.5);
	\draw[dashed] (8,0) arc (0:30:8);
	\draw[dashed] (9.5,0) arc (0:30:9.5);
	\draw[dashed] (3.5*0.996194698,3.5*0.08715) -- (11*0.996194698,11*0.08715);
	\draw[dashed] (3.5*0.984807753,3.5*0.173648) -- (11*0.984807753,11*0.173648);
	\draw[dashed] (3.5*0.96592582,3.5*0.258819045) -- (11*0.96592582,11*0.258819045);
	\draw[dashed] (3.5*0.93969262,3.5*0.3420201433) -- (11*0.93969262,11*0.3420201433);
	\draw[dashed] (3.5*0.90630778,3.5*0.4226182617) -- (11*0.90630778,11*0.4226182617);

	\tkzDefPoint(5,0){P1}
	\tkzDefPoint(7.5,0){P2}
	\tkzDefPointBy[rotation=center P1 angle 30](P2)
	\tkzGetPoint{P3}
	\tkzDrawVector[line width=1pt](P1,P3)
	\tkzDrawPoint[size=6pt](P1)
	\tkzMarkAngle[fill=blue!25,mkpos=.2, size=0.8](P2,P1,P3)
	\tkzLabelAngle[pos=1](P2,P1,P3){$\alpha$}
	
	\tkzDefPoint(8.95,3.25){P1}
	\tkzDefPoint(11.45,3.25){P2}
	\tkzDefPointBy[rotation=center P1 angle 20](P2)
	\tkzGetPoint{P2}
	\tkzDrawSegment[dashed](P1,P2)
	\tkzDefPointBy[rotation=center P1 angle 30](P2)
	\tkzGetPoint{P3}
	\tkzDrawPoint[size=6pt](P1)
	\tkzDrawVector[line width=1pt](P1,P3)
	\tkzMarkAngle[fill=blue!25,mkpos=.2, size=0.8](P2,P1,P3)
	\tkzLabelAngle[pos=1](P2,P1,P3){$\alpha$}

	\tkzLabelSegment[below](A1,A){$b$};
	\tkzLabelSegment[above](B1,B){$b$};
	\tkzText(3,0.8){$a$};
	\tkzText(11.5,2.7){$a+\e$};
\end{tikzpicture}
\end{center}
\caption{The limit manifold $\N$, its partition into sub-sectors $\N_n(i,j)$ and its path-independent parallel transport.}
\label{fig:M}
\end{figure}

We endow $\N$ with a polar system of coordinates $(r,\vp)$, 
\[
(r,\vp) \in [R_0,R_1]\times[0,\e/b].
\]
In these coordinates the Euclidean metric takes the form
\[
\g = dr\otimes dr + r^2\,d\vp\otimes d\vp.
\]
We further endow $T\N$ with a connection $\nabla$, defined by declaring the frame field 
$E = (\pl_r, r^{-1}\,\pl_\vp)$ parallel. We denote by $\Pi$ the (path-independent) parallel transport operator of $\nabla$,
\[
\Pi_p^q:T_p\N\to T_q\N.
\]  
Since $E$ is orthonormal, it follows that $\nabla$ is metrically-compatible \Cy{(see e.g. \cite{DoC92} p.53)}. Note however that $E$ is not parallel with respect to the Levi-Civita connection, hence $\nabla$ is not the Levi-Civita connection, i.e., it is non-symmetric and carries torsion.  
A direct calculation shows that this torsion equals
\[
T = \frac{1}{r}\, dr\wedge d\vp\otimes \partial_\vp.
\]

Note that constant-$r$ and constant-$\vp$ parametric curves are, by definition, $\nabla$-geodesics, but only constant-$r$ curves are locally length-minimizing. 
See Figure~\ref{fig:M} for an illustration of how vectors are parallel transported under $\Pi$.
Note that $\nabla$ admits, by definition, a global parallel frame field, hence its curvature tensor is zero. Since it is also metrically-consistent and non-symmetric, the triplet $(\N,\g,\nabla)$ is a Weitzenb\"ock manifold.

Our main result can be stated as follows:
\begin{quote} 
\emph{
The sequence of locally-Euclidean smooth manifolds with connections $(\M_n,\g_n,\nabla_n)$ converges to the Euclidean sector with non-symmetric connection $(\N,\g,\nabla)$.
}
\end{quote}
The  mode of convergence will be described below.
In Subsection~\ref{subsec:GH} we prove the convergence of $(\tM_n,\g_n)$ to $(\N,\g)$ in the Gromov-Hausdorff (GH) sense. In Subsection~\ref{subsec:Conn} we construct homeomorphisms
\[
F_n : \N \to \tM_n, 
\]
that (i) realize the GH convergence, i.e., have asymptotically vanishing distortions, and (ii) have the property that pullbacks of parallel frame fields of $(\M_n,\nabla_n)$ converge to a parallel frame field of $(\N,\nabla)$.
We then prove some properties of these homeomorphisms, which guide us in the definition of a general notion of convergence described in Section~\ref{sec:conv}.

\subsection{Gromov-Hausdorff convergence}
\label{subsec:GH}

The GH distance is a measure of \Cy{distance} between metric spaces, and is a metric on isometry classes of compact metric spaces (\cite{Pet06}, Chapter 10). A sufficient and necessary condition for a sequence of metric spaces $(Z_n,d_n)$ to converge in the GH sense to a metric space $(Z,d)$ is that there exist bijections
\[
T_n: A_n\to B_n,
\]
where $A_n\subset Z$ and $B_n\subset Z_n$ are finite $\delta_n$-nets, $\delta_n\to0$, and the distortion of $T_n$,
\[
\dis T_n = \max_{x,y\in A_n} |d(x,y) - d_n( T_n(x), T_n(y))|,
\]
tends to zero.

\begin{theorem}
\label{th:GH}
Let $(\tM_n,\g_n)$ be the sequence of compact metric spaces defined in Section~\ref{sec:sequence}, and let $(\N,\g)$ be the Euclidean sector defined above. Then,
$(\tM_n,\g_n)$ GH converges to $(\N,\g)$.
\end{theorem}

\begin{proof}
Denote by $X_n$ the union of boundaries of the $n^2$ blocks forming $\tM_n$ ($X_n$ is the union of both dashed and solid lines in Figure~\ref{fig:Mk}). 
The vertices of $X_n$ form a finite $O(n^{-1})$-net of $\tM_n$. 

Given $n$, we partition $\N$ into $n^2$ sectors, where the $(i,j)$-th sector, denoted by $\N_n(i,j)$ is of type
\[
\N_n(i,j) = \N\brk{a_{n,i}, b_n, \e_n},
\]
where $a_{n,i},b_n,\e_n$ are defined in \eqref{eq:abeps}.
In polar coordinates, 
\[
\N_n(i,j) =
[r_i,r_{i+1}]\times[\vp_j, \vp_{j+1}],
\]
where
\[
r_i=(i-1)\,\Delta R_n,\qquad\qquad
 \vp_j = (j-1)\, \Delta \vp_n.
\]
and
\[
\Delta R_n = \frac{R_1-R_0}{n} = b_n,
\qquad\qquad
\Delta\vp_n = \e_n/b_n.
\]
In correspondence with $X_n$, we denote by $Y_n$ the union of the boundaries of $\N_n(i,j)$
($Y_n$ is the union of both dashed and solid lines in Figure~\ref{fig:M}).

These partitions of $\N$ and $\tM_n$ have the following properties:
\begin{enumerate}
\item The vertices of $Y_n$ form a finite $O(n^{-1})$-net  of $\N$ and have the same cardinality as the vertices of $X_n$.
\item The boundary of $\N_n(i,j)$ consists of curves that are of the same length as the boundary of $\M_n(i,j)$.
\item $Y_n$ consists of $\nabla$-geodesics and $X_n$ consists of $\nabla_n$-geodesics.  

\end{enumerate}

It follows that there exists a natural mapping $T_n: Y_n\to X_n$ that preserves the intrinsic distance of $Y_n$ and $X_n$ (the intrinsic distances on path-connected subsets differ from the induced distances $d$ and $d_n$).
In particular, $T_n$ restricted to the vertices of $Y_n$ is a bijection between two finite $O(n^{-1})$-nets of $\M_n$ and $\N$ respectively.  To prove that $(\tM_n,\g_n)$ converges to $(\N,\g)$ in the GH sense it only remains to show that
\[
\dis T_n \to 0,
\]
where the distortion is with respect to the induced distances $d$ and $d_n$.

The proof relies on two lemmas, whose proofs are given in the appendix.
The first lemma shows that the restrictions of $T_n$ to the boundaries $\partial\N_n(i,j)$ of single cells,
has a distortion of order $O(n^{-2})$:

\begin{lemma}
\label{lm:2}
Let $a,b,\e>0$ and $\theta\in(0,\pi/2)$ be given. 
Let $T_{n,i,j}$ be the natural intrinsic distance preserving mapping,
\[
T_{n,i,j} : \partial \N_n(i,j) \to \partial \tM_n(i,j).
\] 
Then, there exists a constant $c>0$ independent of $n,i,j$, 
such that
\[
\max_{p,q\in \partial\N_n(i,j)} \left|d(p,q) - d_n(T_{n,i,j}(p),T_{n,i,j}(q))\right| < \frac{c}{n^2}.
\]
\end{lemma}

In other words, since $\e_n$ tends to zero faster than $a_{n,i},b_n$, both $\M_n(i,j)$ and $\N_n(i,j)$ become metrically similar to a Euclidean rectangle of size $a_{n,i}\times b_n$, and in particular to each other. Lemma~\ref{lm:2} quantifies this assertion.

The second lemma bounds the number of cells intersected by a length minimizing curve, thus allowing to estimate the accumulated distortion along such a curve:

\begin{lemma}
\label{lm:3}
For every  $n\in\bbN$ and $p,q\in Y_n$, the shortest path in $\N$ connecting $p$ and $q$ intersects at most $3n$ out of the $n^2$ sectors $\N_n(i,j)$. 
Likewise, for every  $n\in\bbN$ and $p,q\in X_n$, the shortest path in $\tM_n$ (viewed as a metric space) connecting $p$ and $q$ intersects at most $3n$ out of the $n^2$ ``rectangles" $\tM_n(i,j)$. 
\end{lemma}


To complete the proof of the theorem,
let $p,q\in Y_n$, and let $\gamma:[0,1]\to \N$ be the shortest path in $\N$ connecting $p$ and $q$, i.e.,
\[
\text{Length}(\gamma) = d(p,q).
\] 
Denote by $p=p_0,p_1,\ldots,p_m=q$ the entrance/exit points of $\gamma$ into sectors $\N_n(i,j)$  in $Y_n$, that is, $p_j=\gamma(t_j)$ where $(t_0,t_1,\dots,t_m)$ is the coarsest partition of $[0,1]$ for which $\gamma$ maps each interval into a single sector.
By Lemma \ref{lm:3}, $m\le 3n$, whereas by Lemma \ref{lm:2} 
\[
|d(p_j,p_{j+1})-d_n(T_n(p_j) , T_n(p_{j+1}))| <  c\,n^{-2}.
\]
Hence, there exists a curve $\sigma:[0,1]\to \tM_n$ such that $\sigma(t_j) = T_n(p_j)$ and
\[
\text{Length}(\sigma) < \text{Length}(\gamma) + 3n \cdot c\,n^{-2}
\]
(see Figure~\ref{fig:gamma}).
It follows that
\[
d_n(T_n(p), T_n(q)) < d(p,q) + O(n^{-1}).
\]
A similar argument, starting from a curve connecting $T_n(p)$ to $T_n(q)$ shows that
\[
d(p,q) < d_n(T_n(p), T_n(q)) + O(n^{-1}),
\]
hence
\[
\dis T_n = O(n^{-1}),
\]
which completes the proof.

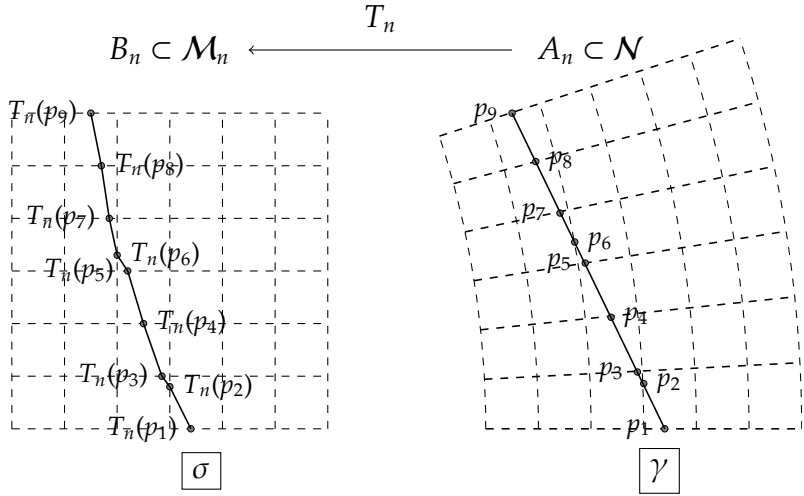
\begin{figure}
\begin{center}
\begin{tikzpicture}[scale=0.7]
	\tkzInit[xmin=-1, xmax=15, ymin=-1.3, ymax=8.5]
	\tkzClip
	\foreach \i in {1,...,7}
	{
		\draw[dashed] (0,\i-1) -- (6,\i-1);
		\draw[dashed] (\i-1,0) -- (\i-1,6);
	}

	\tkzDefPoint(-9,0){O}
	\tkzDefPoint(9,0){P1}
	\tkzDefPointBy[rotation=center O angle 3](P1);	\tkzGetPoint{P2}
	\tkzDefPointBy[rotation=center O angle 3](P2);	\tkzGetPoint{P3}
	\tkzDefPointBy[rotation=center O angle 3](P3);	\tkzGetPoint{P4}
	\tkzDefPointBy[rotation=center O angle 3](P4);	\tkzGetPoint{P5}
	\tkzDefPointBy[rotation=center O angle 3](P5);	\tkzGetPoint{P6}
	\tkzDefPointBy[rotation=center O angle 3](P6);	\tkzGetPoint{P7}

	\tkzDefPoint(15,0){Q1}
	\tkzDefPointBy[rotation=center O angle 3](Q1);	\tkzGetPoint{Q2}
	\tkzDefPointBy[rotation=center O angle 3](Q2);	\tkzGetPoint{Q3}
	\tkzDefPointBy[rotation=center O angle 3](Q3);	\tkzGetPoint{Q4}
	\tkzDefPointBy[rotation=center O angle 3](Q4);	\tkzGetPoint{Q5}
	\tkzDefPointBy[rotation=center O angle 3](Q5);	\tkzGetPoint{Q6}
	\tkzDefPointBy[rotation=center O angle 3](Q6);	\tkzGetPoint{Q7}

	\foreach \i in {1,...,7}
	{
		\tkzDrawSegment[dashed](P\i,Q\i);
		\draw[dashed] (8+\i,0) arc (0:18:17+\i);
	}

	\node at (3,7.2) {$B_n\subset\M_n$};
	\node at (11,7.2) {$A_n\subset\N$};
	\draw[->] (9.5,7.2) -- (4.5,7.2);
	\node at (7,7.8) {$T_n$};
	\node[draw] at (12.3,-0.8) {$\gamma$};
	\node[draw] at (3.6,-0.8) {$\sigma$};

	\tkzDefPoint(12.4,0){P1}
	\tkzDefPoint(9.5,6){P9}
	\tkzDefPoint(12,0.86){P2}
	\tkzDefPoint(11.88,1.08){P3}
	\tkzDefPoint(11.38,2.12){P4}
	\tkzDefPoint(10.89,3.15){P5}
	\tkzDefPoint(10.69,3.55){P6}
	\tkzDefPoint(10.41,4.1){P7}
	\tkzDefPoint(9.95,5.08){P8}
	\tkzDrawPoints(P1,P2,P3,P4,P5,P6,P7,P8,P9)
	\tkzDrawSegment(P1,P9)
	\foreach \i in {1,3,5,7,9}
	{
		\tkzLabelPoint[left=1pt](P\i){{\footnotesize $p_\i$}}
	}
	\foreach \i in {2,4,6,8}
	{
		\tkzLabelPoint[right=1pt](P\i){{\footnotesize $p_\i$}}
	}

	\tkzDefPoint(3.4,0){Q1}
	\tkzDefPoint(1.5,6){Q9}
	\tkzDefPoint(3,0.8){Q2}
	\tkzDefPoint(2.85,1){Q3}
	\tkzDefPoint(2.5,2){Q4}
	\tkzDefPoint(2.2,3){Q5}
	\tkzDefPoint(2,3.3){Q6}
	\tkzDefPoint(1.85,4){Q7}
	\tkzDefPoint(1.7,5.){Q8}
	\tkzDrawPoints(Q1,Q2,Q3,Q4,Q5,Q6,Q7,Q8,Q9)
	\tkzDrawSegment(Q1,Q2)
	\tkzDrawSegment(Q2,Q3)
	\tkzDrawSegment(Q3,Q4)
	\tkzDrawSegment(Q4,Q5)
	\tkzDrawSegment(Q5,Q6)
	\tkzDrawSegment(Q6,Q7)
	\tkzDrawSegment(Q7,Q8)
	\tkzDrawSegment(Q8,Q9)
	\foreach \i in {1,3,5,7,9}
	{
		\tkzLabelPoint[left=1pt](Q\i){{\footnotesize $T_n(p_\i)$}}
	}
	\foreach \i in {2,4,6,8}
	{
		\tkzLabelPoint[right=1pt](Q\i){{\footnotesize $T_n(p_\i)$}}
	}

\end{tikzpicture}
\end{center}
\caption{The curves $\gamma$ and $\sigma$ used in the proof of Theorem~\ref{th:GH}.}
\label{fig:gamma}
\end{figure}

\end{proof}
\bigskip

By the nature of the GH distance, the limit of $(\tM_n,\g_n)$ is unique up to isometry of metric spaces. That is, if $(\tM_n,\g_n)$ GH converges also to $(\N',\g')$, then there exists a distance-preserving bijection $(\N,d) \to (\N',d')$. By the Myers-Steenrod theorem (\cite{Pet06}, p.147), this map is smooth and is a Riemannian isometry $(\N,\g)\to(\N',\g')$. 

\subsection{Convergence of the parallel transport}
\label{subsec:Conn}

GH convergence of metric spaces is a very weak notion of convergence. To wit, the sequence of finite metric spaces consisting of the vertices of $X_n$ alone with the induced metric $d_n$ GH-converges to the smooth Riemannian manifold $(\N,\g)$. 
On the other hand, stronger notions of convergence of smooth manifolds, such as H\"older convergence, require $\M_n$ to be diffeomorphic to $\N$, which is not the case. Thus, we look for a new notion of convergence, which is strong enough to capture the smooth structure of $\M_n$ and its parallel transport, while being weak enough to allow for topological defects.

Since the manifolds $\tM_n$ and $\N$ are homeomorphic, it is natural to relate their structures by constructing a sequence of homeomorphisms 
\[
F_n:\N\to \tM_n,
\]
which are smooth on the pre-image of $\M_n$. Moreover, by defining the $F_n$ to be extensions of the $T_n$ defined in the previous section, we guarantee the preservation of both length and  geodesic properties along the $\nabla$-geodesic grids $Y_n$.   At this point the limiting connection $\nabla$ may look arbitrary. In Section~\ref{sec:conv} we will see that it is determined uniquely. 

We define the mappings $F_n$ thought their restrictions to sub-sectors,
\[
F_n: \N_n(i,j) \to \tM_n(i,j).
\]
Recall that the parametrization of $\N_n(i,j)$ by polar coordinates is
\[
\N_n(i,j) = [r_i,r_{i+1}]\times[\vp_j, \vp_{j+1}],
\]
where $r_i = (i-1)\Delta R_n$ and $\vp_j = (j-1)\Delta\vp_n$. 
Recall, furthermore, that $\tM_n(i,j)$ can be represented as two hexagons glued together (see Figure~\ref{fig:R2}). 
$F_n$ maps the lower half, $\vp_j \le \vp \le \vp_{j +1/2}$, of $\N_n(i,j)$ onto the lower hexagon of $\tM_n(i,j)$. Endowing this lower hexagon with canonical Euclidean coordinates, $F_n(r,\vp) = (X(r,\vp),Y(r,\vp))$ is defined by  
\[
X(r,\vp) = r - r_i,
\]
and
\[
Y(r,\vp) = (\vp - \vp_j) \times
\begin{cases}
r_i & 		r\in \Brk{r_i, r_{i+1/2-D/2}}, \\
r_i + \frac{1}{D}\brk{r - r_{i+1/2-D/2}} & 
r\in \Brk{r_{i+1/2-D/2},r_{i+1/2+D/2}},\\
r_{i+1} &	r\in \Brk{r_{i+1/2+D/2},r_{i+1}},
\end{cases}
\]
where $D= \frac{\Delta \vp_n}{2 \tan \theta}=O(n^{-1})$.
The mapping of the upper half, $\vp_{j+1/2} \le \vp \le \vp_{j+1} $, of $\N_n(i,j)$ onto the upper hexagon of $\tM_n(i,j)$ is defined  similarly.
One can verify that $F_n$ is indeed a homeomorphism that extends the mapping $T_n$. 

$F_n$ is a local diffeomorphism everywhere in $F_n^{-1}(\M_n)\subset \N$, except along the lines $r = r_{i+1/2\pm D/2}$ in every sector. We now calculate the derivatives of $F_n$ and the pullback metric on $\N$.
Differentiating $F_n$, 
\[
\frac{\pl X}{\pl r} = 1, \qquad \frac{\pl X}{\pl \vp} = 0,
\]
\[
\frac{\pl Y}{\pl r} = 
\begin{cases}
0 & 							
r\in \brk{r_i, r_{i+1/2-D/2}} \\
\frac{\vp - \vp_j}{D} & 
r\in \brk{r_{i+1/2-D/2} ,r_{i+1/2+D/2}}\\
0 &							
r\in \brk{r_{i+1/2+D/2},r_{i+1}},
\end{cases}
\]
and
\begin{align}
\label{eq:dF_n}
\frac{\pl Y}{\pl \vp} = 
\begin{cases}
r_i & 						
r\in \brk{r_i, r_{i+1/2-D/2}} \\
r_i + \frac{1}{D}\brk{r - r_{i+1/2-D/2}} & 
r\in \brk{r_{i+1/2-D/2},r_{i+1/2+D/2}}\\
r_{i+1} &					
r\in \brk{r_{i+1/2+D/2},r_{i+1}}.
\end{cases}
\end{align}
This mapping can be slightly modified to be $C^1$ (and even smooth) in $F_n^{-1}({\M_n})$, resulting in a diffeomorphism $F_n^{-1}({\M_n})\to \M_n$.

We now prove several properties of the mappings $F_n$ that will be relevant for the notion of convergence defined in Section~\ref{sec:conv}. Proposition~\ref{pn:vanishing dis} deals with the vanishing distortion of $F_n$. Proposition~\ref{pn:connection convergence} deals with the convergence of the pullback connections.

\begin{proposition}\ 
\label{pn:vanishing dis}
\begin{enumerate}
\item $\dis F_n \to 0$.
\item For every $p\in[1,\infty)$,
\[
\int_{\N} \dist{^p}(dF_n,\SO{\g,\g_n}) \,\Volume \to 0,
\]
where $\SO{\g,\g_n}$ denotes the set of metric- and orientation-preserving linear maps $T\N\to F_n^*T\M_n$.
\end{enumerate}
\end{proposition}

\begin{proof}
Item 1 follows from the fact that $F_n$ is an extension of $T_n$. 
Item 2 follows from \eqref{eq:dF_n}, since $dF_n$ tends uniformly to $\SO{\g,\g_n}$ on the domain 
\[
\bigcup_{i,j=1}^n \BRK{(r,\vp)\in \N_n(i,j) \,\,:\,\, r\notin (r_{i+1/2-D/2},r_{i+1/2+D/2})},
\]

and the area of its complement, where $dF_n$  is uniformly bounded, tends to zero.
\end{proof}

\begin{proposition}
\label{pn:connection convergence}
Denote by $E_n$ the frame field on $\M_n$ generated by the vector fields $(\pl_X,\pl_Y)$ on the Euclidean hexagons (it is an orthonormal parallel frame field of the Levi-Civita connection $\nabla_n$ on $\M_n$). Denote by $E$ the frame field $(\pl_r,r^{-1}\pl_\vp)$ on $\N$ (it is an  orthonormal parallel frame field  of the connection $\nabla$ on $\N$). Then
$F_n^\star E_n\to E$ in $L^p$ for every $p\in[1,\infty)$, 
\begin{align}
\label{eq:frame convergence}
\limn \int_\N |F_n^\star E_n - E|^p_{\g}\,\Volume = 0,
\end{align}
where the norm of a tuple of vector fields is the sum of the norms.
Furthermore,  $F_n^\star E_n\to E$ almost everywhere.
In particular, since $E_n$ and $E$ are orthonormal and parallel, the parallel transport  from $x$ to $y$ with respect to $F_n^\star \nabla_n$ converges to the parallel transport from $x$ to $y$ with respect to $\nabla$, for almost every $x,y\in \N$.
\end{proposition}

\begin{remark}
This is a weak form of convergence of the connection, which entails the convergence of the parallel transport operator, but not the convergence of the derivative operator. In particular, the Christoffel symbols of the pullback connection $F_n^\star \nabla_n$ do not converge to the Christoffel symbols of $\nabla$. In fact, since the mappings $(r,\vp)\mapsto (X,Y)$ are eventually almost everywhere affine, the Christoffel symbols converge  almost everywhere to $0$ pointwise (which are the symbols of the Levi-Civita connection, and not of $\nabla$).
\end{remark}

\begin{proof}
From equation \eqref{eq:dF_n},
\[
F_n^\star E_n = dF_n^{-1}(\pl_X,\pl_Y) = 
\begin{cases}
(\pl_r, r_i^{-1} \pl_\vp) & 						
r\in \brk{r_i, r_{i+1/2-D/2}}, \\
\cdots & r\in \brk{r_{i+1/2-D/2},r_{i+1/2+D/2}},\\
(\pl_r, r_{i+1}^{-1} \pl_\vp) &					
r\in \brk{r_{i+1/2+D/2},r_{i+1}},
\end{cases}
\]
where $\cdots$ in the middle term stands for a frame uniformly bounded in $n$. Since $D=O(n^{-1})$, the almost everywhere convergence and equation \eqref{eq:frame convergence} follow immediately.
\end{proof}

Proposition~\ref{pn:vanishing dis} asserts that the distortion of $F_n$ vanishes -- this is a global claim -- and that locally, $dF_n$ is asymptotically rigid in the mean.
Even though this is not relevant to the subsequent analysis, there is  more to be said about the  mappings $F_n$, and specifically, on the convergence of the pullback metrics $F_n^\star \g_n$ to the Euclidean metric  $\g$ on $\N$.

We conclude this section by stating several of these results, both for the sake of completeness, and since they provide a better understanding of how the sequence $(\M_n,\g_n)$ converges to $(\N,\g)$. All of them follow from direct calculations using \eqref{eq:dF_n}.
\begin{proposition}\
\begin{enumerate}
\item $F_n^\star \g_n \to \g$ in $L^p$ for every $p\in[1,\infty)$, 
\[
\limn \int_{\N} |F_n^\star\g_n - \g|_\g^p\, \Volume = 0.
\]
Furthermore, $F_n^\star \g_n \to \g$ almost everywhere. 
By smoothing $F_n$ we can obtain pointwise convergence in $F_n^{-1}(\M_n)$.

\item It follows from the previous item that for every vector field $X\in \Gamma(T\N)$,
\[
\limn \int_\N |X|_{F_n^\star\g_n}^p \Volume = \int_\N |X|_{\g}^p \Volume
\]
\item $F_n^\star\g_n$ does \emph{not} converge to $\g$ uniformly or in $L^\infty$ (even for smooth modifications of $F_n$).
\item The volume form $\Vol_{F_n^\star\g_n}$  converges to $\Volume$ in $L^\infty$. In particular, the induced measures 
$\mu_{F_n^\star\euc}$ 
converge 
to $\mu_\g$ in total variation.
This is, in a sense, a volume-equivalent of  vanishing distortion. It also shows that $\M_n$ converge to $\N$ as metric measure spaces.
\end{enumerate}

\end{proposition}

\section{Convergence of manifolds with defects}
\label{sec:conv}

In Section~\ref{sec:conv_example} we constructed a sequence $F_n:\N\to\tM_n$ of homeomorphisms, which are diffeomorphisms on $F_n^{-1}(\M_n)$. We showed that they have asymptotically vanishing distortions, they are asymptotically rigid in the mean, and that there exist $\nabla_n$-parallel frame fields $E_n$, whose pullback $F_n^\star E_n$ converge in the mean to a $\nabla$-parallel frame field $E$.

A natural question is whether  
the sequence $(\M_n,\g_n,\nabla_n)$ defines a unique limit $(\N,\g,\nabla)$.
The metric limit is clearly unique (modulo Riemannian isometries) by the properties of GH convergence and the Myers-Steenrod theorem. It is not yet clear, however, whether 
a limit connection is well-defined. In Section~\ref{sec:conv_example} we characterized the convergence of a sequence of flat connections $\nabla_n$ through the convergence of pullbacks of $\nabla_n$-parallel frame fields. For such a mode of convergence to be unambiguous, we have to prove that any sequence of asymptotically rigid maps $\N\to\M_n$ with asymptotically vanishing distortion and for which the pullback of parallel frame fields converges, results in the same limiting connection.

In order to prove that $(\tM_n,\g_n)$ GH-converges to $(\N,\g)$, it is sufficient to examine the distortion associated with mappings between nets. 
Similarly, it is possible to define a convergence of connections by examining mappings from subsets of $\N$ to subsets of $\M_n$, excluding sets of asymptotically vanishing volume. 
We will exclude from $\M_n$ asymptotically small neighborhoods of the lines $\tM_n\setminus \M_n$. 
In other words, manifolds with singularities are replaced by manifolds with asymptotically small  ``holes".
The advantage of this approach is that we are then in the realm of diffeomorphisms between compact smooth manifolds with corners, and do not have to deal with singularities, nor with a lack of compactness. 

The following definition establishes a notion of weak convergence of Weitzenb\"ock manifolds, that is, Riemannian manifolds endowed with metrically-consistent (i.e. metric) locally-flat connections.

\begin{definition}
\label{df:convergence}
Let $(\calM_n,\g_n,\nabla_n)$, $(\calM,\g,\nabla)$ be compact \Cy{oriented} $d$-dimensional Weitzenb\"ock manifolds with corners.
We say that the sequence $(\calM_n,\g_n,\nabla_n)$ converges to $(\calM,\g,\nabla)$ 
with $p\in \Cy{[}d, \infty)$, if there exists a sequence of diffeomorphisms   $F_n: A_n\subset \M\to \M_n$ such that:
\begin{enumerate}
\item $A_n$ covers $\M$ asymptotically:
\[
\limn \textVol_\g (\M\setminus A_n ) = 0.
\]
\item $F_n$ are approximate isometries: the distortion vanishes asymptotically, namely,
\[
\limn\dis F_n = 0.
\]
\item $F_n$ are asymptotically rigid in the mean:
\[
\limn \int_{A_n} \dist{^p}(dF_n,\SO{\g,\g_n}) \,\Volume = 0.
\]
\item The parallel transport converges in the mean in the following sense: every point in $\M$ has a neighborhood $U\subset\M$, with (i) a $\nabla$-parallel frame field $E$ on $U$, and (ii) a sequence of $\nabla_n$-parallel frame fields $E_n$  on $F_n(U\cap A_n)$,
such that
\[
\limn \int_{U\cap A_n} |F_n^\star E_n-E|^p_\g \Volume = 0.
\]
\end{enumerate}
\end{definition}

\begin{corollary}
\label{cy:conv}
The sequence of manifolds with defects $(\M_n,\g_n,\nabla_n)$ defined in Section~\ref{sec:sequence}, converges in the sense of  Definition~\ref{df:convergence}  to the Euclidean sector with connection $(\N,\g,\nabla)$. 
\end{corollary}

\begin{proof}
This follows from Propositions~\ref{pn:vanishing dis}-\ref{pn:connection convergence}.
To comply with Definition~\ref{df:convergence} one has to take $\M_n$ to be compact manifolds by removing asymptotically small neighborhoods around the singular lines $\tM_n\setminus \M_n$, and restrict the functions $F_n$ accordingly. 
It is immediate that Proposition~\ref{pn:connection convergence} holds after the restriction of $F_n$.
To show that Proposition~\ref{pn:vanishing dis} also holds, observe that the dislocation lines in $\tM_n$ are of length $O(n^{-2})$.
Therefore, it is possible to remove neighborhoods of diameter $O(n^{-2})$ around the singularity lines, thus changing the distance functions only by $O(n^{-2})$. 
Lemmas~\ref{lm:2}-\ref{lm:3} still hold after the removal of these neighborhoods, from which  Theorem~\ref{th:GH}, and therefore Proposition~\ref{pn:vanishing dis}, follow.
\end{proof}

\bigskip

The following theorem shows that the convergence of sequences of Riemannian manifolds with connections is well-defined: the limit  is {unique up to isomorphisms.}

\begin{theorem}
\label{th:uniqueness}
Let $(\calM_n,\g_n,\nabla_n)$, $(\M,\g,\nabla^\M)$ and $(\N,\mathfrak{h},\nabla^\N)$ be 
compact Riemannian manifolds with corners,  endowed with metrically-consistent locally-flat connections.
Suppose that 
\[
(\M_n,\g_n,\nabla_n) \to (\M,\g,\nabla^\M)
\Textand 
(\M_n,\g_n,\nabla_n) \to (\N,\mathfrak{h},\nabla^\N)
\]
in the sense of Definition~\ref{df:convergence}. Then, there exists a Riemannian isometry $H:\M\to\N$ such that $H^\star \nabla^\N = \nabla^\M$.
\end{theorem}

Since the proof is long and  technical, we start by presenting a sketch. 
By definition, there exist sequences of diffeomorphisms
\[
F_n:A_n\subset \M\to{\M}_n
\Textand
G_n: B_n\subset \N\to{\M}_n
\]
that are approximate isomorphisms of Riemannian manifolds with connections in the sense of Definition~\ref{df:convergence}. 
\Cy{Note that since $F_n$ and $G_n$ are diffeomorphisms, Item 3 in Definition~\ref{df:convergence} implies that $F_n$ and $G_n$ are, for $n$ large enough, orientation preserving, which we will assume from now on.}
We define
\[
H_n=G_n^{-1}\circ F_n:A_n\to B_n,
\]
which are diffeomorphisms.
It follows from $\dis F_n\to0$ and $\dis G_n\to0$ that 
\[
\lim_{n\to\infty} \dis H_n = 0
\]
as well. We then proceed as follows:

\begin{enumerate}
\item Lemma~\ref{lm:metric uniqueness}, the metric part: it follows from the properties of GH convergence that $(\M,\g)$ and $(\N,\h)$ are isometric. We show that there exists a Riemannian isometry, which we denote by $H:\M\to\N$, which is the uniform limit of a subsequence of the maps $H_n$. In the rest of the proof we show that $H$ satisfies $H^\star \nabla^\N = \nabla^\M$.

\item
Lemma~\ref{lm:restriction}: The convergence of the connections in Definition~\ref{df:convergence} is associated with the convergence of pullbacks of \emph{local} frame fields.  We show that we can restrict ourselves to neighborhoods that admit \emph{global} frame fields. Hence, we may assume, without loss of generality, the existence and convergence of pullbacks of global frame fields.

\item
Lemma~\ref{lm:uniqueness}: We show that the limit of a specific sequence of (global) frame fields is unique in the following sense: if $E_n$ are frame fields on $\M_n$ such that $F_n^\star E_n$ converges to $E^\M$ and $G_n^\star E_n$ converges to $E^\N$ (both in $L^p$), then $H_\star E^\M = E^\N$.

\item
Lemma~\ref{lm:uniqueness2}: We complete the proof by showing that if $E_n$ and $D_n$ are frame fields on $\M_n$ such that $F_n^\star E_n$ converges to $E^\M$ and $G_n^\star D_n$ converges to $E^\N$ (both in $L^p$), then $H_\star E^\M$ is a $\nabla^\N$-parallel frame field, hence $H^\star \nabla^\N = \nabla^\M$.

\end{enumerate}

A comment about notations: 
here and throughout this paper, we consider  differentials such as $dF_n$ as maps 
$T\M\to F_n^*T\M_n$, where $F_n^*T\M_n$ is a vector bundle over $\M$, with
the fiber $(F_n^*T\M_n)_p$ identified with the fiber $T_{F_n(p)}\M_n$. The differential should be 
distinguished from the push-forward operator for vector fields, $(F_n)_\star:T\M\to T\M_n$.
Likewise, we 
denote by $F_n^*$ the pullbacks of  vector fields and differential forms, both considered as sections of $T\M_n$ or $T^*\M_n$. This should not be confused with the closely related pullback involving composition with $dF_n$, which we denote by $F_n^\star$.
That is, if $X$ is a vector field on $\M_n$, then $F_n^*X$ is a section of $F_n^*T\M_n$, with $F_n^*X(p)=X(F_n(p))$, 
whereas $F_n^\star X$ is a section of $T\M$, where $F_n^\star X(p) = dF_n^{-1} \circ X(F_n(p))$.

\begin{lemma}
\label{lm:metric uniqueness}
There exists an isometry $H:(\M,\g)\to(\N,\h)$, which is the uniform limit of a (not relabeled) subsequence $H_n$ in the sense that 
\[
\sup_{p\in A_n} d_\N(H_n(p),H(p)) \to 0 \textand  \sup_{q\in B_n} d_\M(H_n^{-1}(q),H^{-1}(q)) \to 0.
\]
\end{lemma}

\begin{proof}
Since by Item~1 in Definition~\ref{df:convergence}
\[
\mu_\g (\M \setminus A_n) \to 0 \Textand \mu_\h (\N \setminus B_n) \to 0, 
\]
it follows that $A_n$ and $B_n$ are $\e_n$-nets of $\M$ and $\N$ for some $\e_n \to 0$, i.e., $H_n$ are bijective mappings between $\e_n$-nets.
Since $\dis H_n\to 0$, it follows that the GH distance between $\M$ and $\N$ is zero, hence $(\M,\g)$ and $(\N,\h)$ are isometric as metric spaces. By the Myers-Steenrod theorem this isometry is also a Riemannian isometry.

We now construct a specific isometry $H$. 
We extend the maps $H_n:A_n\to B_n$ into maps $\hH_n:\M\to \N$ with asymptotically vanishing distortion. Since $A_n$ is an $\e_n$-net of $\M$, there exists a map $\psi_n:\M\to A_n$, such that 
\[
\psi_n|_{A_n} = \id
\Textand
\sup_{p\in\M} d_\M(p,\psi_n(p)) < \e_n.
\]
We define 
\[
\hH_n(p) = H_n(\psi_n(p)).
\]
The sequence $\hH_n$ has asymptotically vanishing distortions:
for all $p,p'\in\M$,
\[
\begin{split}
|d_\M(p,p') - d_\N(\hH_n(p),\hH_n(p'))| 
&=|d_\M(p,p') - d_\N(H_n(\psi_n(p)),H_n(\psi_n(p')))| \\
&\hspace{-3cm}\le |d_\M(p,p') -d_\M(\psi_n(p),p')| \\
&\hspace{-2.5cm} +|d_\M(\psi_n(p),p')-d_\M(\psi_n(p),\psi_n(p'))| \\
&\hspace{-2.5cm} + |d_\M(\psi_n(p),\psi_n(p'))- d_\N(H_n(\psi_n(p)),H_n(\psi_n(p')))| \\
&\hspace{-3cm}\le 2\e_n + \dis H_n \to 0.
\end{split}
\]
Note, however, that $\hH_n$ are not  diffeomorphisms: they are neither injective nor surjective, and may not even be continuous (depending on the choice of $\psi_n$).

Let $\calA\subset \M$ be a dense countable subset.
Via a standard Arzela-Ascoli argument, there exists a subsequence (not relabeled) such that $\hH_n(a_k)$ converges for every $k$. 
Denote the resulting function by $H:\calA\to \N$,
\[
H(a) = \limn \hH_n(a) \qquad \forall a\in\calA.
\]
Clearly, $\dis H=0$, i.e. $H$ is distance-preserving.
Since $\calA$ is dense in $\M$, $H$ can be extended to a distance-preserving function $\M\to\N$. 
For all $p\in\M$ and $a\in\calA$,
\beq
\label{eq:H_n to H}
\begin{split}
d_\N(\hH_n(p),H(p)) &\le d_\N(\hH_n(p),\hH_n(a))+ d_\N(\hH_n(a),H(a))+d_\N(H(a),H(p)) \\
		&\le \dis \hH_n + 2d_\M(p,a) + d_\N(\hH_n(a),H(a)).
\end{split}
\eeq
Let $\e>0$ be given. Let $\{a_1,\ldots,a_\ell\}\subset \calA$ be a finite $\e/6$-net of $\M$.
Let $N\in\bbN$ be large enough such that for all $n>N$,
\[
\dis \hH_n<\e/3, \Textand
\max_{k=1,\dots,\ell} d_\N(\hH_n(a_k),H(a_k))<\e/3.
\]
By choosing $a$ in \eqref{eq:H_n to H} in the set $\{a_1,\ldots,a_\ell\}$ with $d_\M(p,a)<\e/6$, we obtain that for all $p\in\M$ and all $n>N$,
\[
d_\N(\hH_n(p),H(p))<\e,
\]
i.e., $\hH_n$ converges to $H$ uniformly. Since $\hH_n$ is an extension of $H_n$, it follows that 
\[
\sup_{p\in A_n} \,\,\, d_\N(H_n(p),H(p)) \to 0.
\]

It remains to show that $H$ is surjective.
Similarly to the above construction, extend $K_n =H_n^{-1}$ to mappings $\hK_n:\N\to \M$ satisfying $\dis \hK_n\to 0$. Even though $\hK_n\ne \hH_n^{-1}$ (neither $\hK_n$ nor $\hH_n$ are invertible),

\[
\hK_n\circ\hH_n = \hK_n\circ H_n\circ \psi_n = \psi_n,
\]
where we used the fact that $\hK_n= H_n^{-1}$ on the image of $\psi_n$.
Thus,
\[
d_\M(p,\hK_n\circ \hH_n (p)) = d_\M(p,\psi_n(p)) < \e.
\]

By the same arguments as above, we construct from $\hK_n$ a distance-preserving map $K:\N\to \M$, which is the uniform limit of a subsequence of $\hK_n$,
\[
\sup_{q\in\N} \,\,\, d_\M(\hK_n(q),K(q)) \to 0.
\]
Since
\[
\begin{split}
d_\M(p,K\circ H (p)) &\le d_\M(p, \hK_n \circ \hH_n (p)) + d_\M(\hK_n\circ \hH_n(p), \hK_n \circ H(p)) \\
&\hspace{0.5cm} + d_\M(\hK_n \circ H(p),K\circ H(p))\\
&\le \e_n + \dis \hK_n + d_\N(\hH_n(p),H(p))\\
&\hspace{0.5cm}+ d_\M(\hK_n \circ H(p),K\circ H(p)),
\end{split}
\]
it follows that the right-hand side tends to $0$ as $n\to \infty$, i.e., $K = H^{-1}$.
Thus, $H:\M\to \N$ is a distance-preserving bijection. By the Myers-Steenrod theorem it is a smooth Riemannian isometry.  
\end{proof}

\bigskip

In the remaining of this section we show that $H^\star\nabla^\N = \nabla^\M$. 
Specifically, we show that every point $p\in\M$ has a neighborhood $U$ endowed with a $\nabla^\M$-parallel frame field $E^U$, such that $H$ pushes forward $E^U$ into a $\nabla^\N$-parallel frame field $E^V$ on $V = H(U)$. 

We will show it by proving that Theorem~\ref{th:uniqueness} holds under the assumption that $\nabla_n$, $\nabla^\M$ and $\nabla^\N$ all admit \emph{global} parallel frame fields, and that the isometry that pushes the global frame fields is the uniform limit $H$ of $H_n$ (Lemmas~\ref{lm:uniqueness} and \ref{lm:uniqueness2}). To apply this particular case to the general case, we show that it is possible to restrict $\M$, $\N$ and $\M_n$ to submanifolds $U$, $V$ and $U_n$ that admit global frame fields, such that the convergence of $\M_n$ implies the convergence of $U_n$ (as stated in the following lemma).

\begin{lemma}
\label{lm:restriction}
Every point $p\in \M$ has a compact neighborhood $U\subset \M$, such that 
\begin{align}
\label{eq:restriction}
(U_n,\g_n,\nabla_n) \to (U,\g,\nabla^\M) \textand 
(U_n,\g_n,\nabla_n) \to (V,\mathfrak{h},\nabla^\N),
\end{align}
where $U$, $V=H(U)$ and $U_n = F_n(U\cap A_n)\cap G_n(V\cap B_n)$ all admit global frame fields.
The convergence is realized by restrictions of $F_n$ and $G_n$. 
\end{lemma}

Before proving Lemma~\ref{lm:restriction}, we prove two lemmas. The first  is a geometric version of Hadamard's inequality \cite{Gar07}. The second  shows that $F_n$ and $G_n$ are uniformly close to being rigid over large sets.

\begin{lemma}
\label{lm:Hadamard}
Let $F:(\M,\g)\to (\N,\h)$ be a smooth \Cy{orientation-preserving local-diffeomorphism} between $d$-dimensional \Cy{oriented} Riemannian manifolds, then
\begin{enumerate}
\item
\[
\frac{\Vol_{F^\star\h}}{\Volume} \le |dF|^d,
\]
where 
\[
|dF| = \sup_{0\ne v\in T\M} \frac{|dF(v)|_\h}{|v|_\g}.
\]
\item 
\Cy{ 
\[
\left|\frac{d\textVol_{F^\star \h}}{\Volume}-1\right| \le (\dist (dF,\SO{\g,\h}) + 1)^d - 1.
\]
}
\end{enumerate}
\end{lemma}

\begin{proof}
$dF$ dilates tangent vectors in $T\M$ by at most a factor of $|dF|$, hence at every point $p\in \M$, $dF$ maps a unit $d$-cube in $T_p\M$ (distances are with respect to $\g$) to a $d$-parallelogram in $T_{H(p)}\N$ with edges of length at most $|dF|$  (distances are with respect to $\h$), hence its $\h$-volume is at most $|dF|^d$.
\Cy{This proves the first part.

For the second part, note that when working in local oriented orthonormal frames in $T\M$ and $F^*T\N$, $\dist(dF,\SO{\g,\h})$ is greater or equal to the largest deviation of the singular values of $dF$ from $1$, whereas $d\text{Vol}_{F^\star \h}/\Volume$ is the determinant of $dF$, which is the product of the singular values, since $F$ is orientation preserving. Denote by $r_j$ the singular of $dF$ ($j=1,\ldots,d$), it follows that
\[
\left|\frac{d\textVol_{F^\star \h}}{\Volume}-1\right| = \left|\prod_{j=1}^d r_i -1\right| \le \left|\prod_{j=1}^d (|r_i -1| + 1) -1\right| \le (\dist (dF,\SO{\g,\h}) + 1)^d - 1.
\]}
\end{proof}

\begin{lemma}
\label{lm:A_n^e}
For every $\e>0$ and $n\in\mathbb{N}$ define
\[
A_n^\e = \BRK{ x\in A_n : | d_x F_n |,  |d_{F_n(x)} F_n^{-1}|, |d_{F_n(x)} G_n^{-1}|, | d_{H_n(x)} G_n |<1+\e }.
\] 
Then
\begin{subequations} 
\beq
\label{eq:A_n^e_a}
\lim_{n\to\infty} \textVol_\g (\M\setminus A_n^\e) = 0
\eeq
\beq
\label{eq:A_n^e_b}
\lim_{n\to\infty} \textVol_{\g_n} (\M_n\setminus F_n(A_n^\e)) = 0.
\eeq
\end{subequations} 
\end{lemma}

\begin{proof}
For every $\e>0$ and $n\in \mathbb{N}$, define
\[
C_n^\e = \BRK{ x\in A_n : |d_xF_n|, |d_{F_n(x)}F_n^{-1}| < 1+\e },
\]
and
\[ 
D_n^\e = \BRK{ y\in B_n : |d_yG_n|, |d_{G_n(y)}G_n^{-1}| < 1+\e }.
\]
Note that
\beq
A_n^\e = C_n^\e \cap H_n^{-1}(D_n^\e).
\label{eq:A_n^e_intersect}
\eeq
Since for every $\e<1$ and $x\in\M$,
\[
\dist(d_xF_n,\SO{\g,\g_n}) < \e
\]
implies that
\[
|d_xF_n| < 1 + \e
\Textand
|d_{F_n(x)} F_n^{-1}| < \frac{1}{1-\e},
\]
it follows from items $1$ and $3$ in Definition~\ref{df:convergence} that for every $\e>0$, 
\beq
\label{eq:DnCn}
\lim_{n\to \infty} \textVol_\g(\M\setminus C_n^\e) = 0 \textand \lim_{n\to \infty} \textVol_\h(\N\setminus D_n^\e) = 0.
\eeq
To prove \eqref{eq:A_n^e_a}, \eqref{eq:A_n^e_b} it is sufficient to show that
\beq
\label{eq:DnCn2}
\lim_{n\to \infty} \textVol_{\g_n} (\M_n\setminus G_n( D_n^\e) ) = 0.
\eeq
Indeed, since $|dF_n^{-1}| < 1+\e$ on $C_n^\e$, it follows from Hadamard's inequality (Lemma~\ref{lm:Hadamard}), and Equations \eqref{eq:A_n^e_intersect}, \eqref{eq:DnCn} and \eqref{eq:DnCn2} that
\[
\begin{split}
\textVol_{\g}(\M\setminus A_n^\e) &=
 \textVol_{\g}(\M\setminus (C_n^\e \cap H_n^{-1}( D_n^\e)) ) \\
 &=
	\textVol_{\g}(\M\setminus C_n^\e ) + \textVol_\g (C_n^\e \setminus  H_n^{-1}( D_n^\e)) \\
	& \le \textVol_{\g}(\M\setminus C_n^\e) + (1+\e)^d \textVol_{\g_n} (\M_n \setminus G_n(D_n^\e))\to 0,
\end{split}
\]
which implies \eqref{eq:A_n^e_a}.
If \eqref{eq:DnCn2} holds then by symmetry,
\[
\lim_{n\to \infty} \textVol_{\g_n}(\M_n\setminus F_n ( C_n^\e)) = 0,
\]
and \eqref{eq:A_n^e_b} follows since
\[
F_n(A_n^\e) = F_n(C_n^\e) \cap G_n(D_n^\e).
\]

It remains to prove \eqref{eq:DnCn2}.

\Cy{
Indeed,
\beq
\begin{split}
\textVol_{\g_n} (\M_n\setminus G_n( D_n^\e) ) 
	&= \int_{\M_n\setminus G_n( D_n^\e)} d\textVol_{\g_n} = \int_{B_n\setminus D_n^\e} d\textVol_{G_n^\star \g_n} \nonumber\\
	&\le \int_{B_n\setminus D_n^\e} d\textVol_\h + \int_{B_n\setminus D_n^\e} \left|\frac{d\textVol_{G_n^\star \g_n}}{d\textVol_\h}-1\right| d\textVol_\h \nonumber\\
	&\le \textVol_\h(\N\setminus D_n^\e) + \int_{B_n} \Brk{(\dist(dG_n,\SO{\h,\g_n})+1)^d - 1} d\textVol_\h \nonumber\\
	&\to 0 \nonumber.
\end{split}
\eeq
Where between the second and third lines we used the second part of \lemref{lm:Hadamard}, and in the last lines we used \eqref{eq:DnCn} and Item 3 in \defref{df:convergence} (note that $p\ge d$).
}
\end{proof}

\begin{remark}
Lemma~\ref{lm:A_n^e} is the only place where we use the assumption that $p \Cy{\ge} d$. This assumption can be removed (resulting in $p\in[1,\infty)$) if we add an extra assumption on $F_n$ in Definition~\ref{df:convergence}, requiring the volume forms $\Vol_{F_n^\star}$ to converge in the mean to $\Volume$, or require the convergence of the induced measures.
\end{remark}

\begin{proof}[of Lemma~\ref{lm:restriction}]
Let $p\in \M$, and let $U$ be a neighborhood of $p$ satisfying Item 4 in Definition~\ref{df:convergence} (with respect to $F_n$).
Set $q= H(p)\in \N$ and let $V\subset \N$ be a neighborhood of $q$ satisfying Item  4 in Definition~\ref{df:convergence} (with respect to $G_n$).
Without loss of generality we may assume that $V=H(U)$, otherwise  reduce $U$ to $U\cap H^{-1}(V)$. This choice of neighborhoods ensures that $U$, $V$ and $U_n$ are covered by global parallel frame fields.

The properties in Items~2--4 in Definition~\ref{df:convergence} are  preserved by the restrictions of $F_n$ and $G_n$ to sub-domains $U\cap A_n$ and $V\cap B_n$.
Therefore,
the only non-trivial part of the proof is to show that  $F_n^{-1}(U_n)$ and $G_n^{-1}(U_n)$  cover $U$ and $V$ asymptotically (Item~1). 

Thus, we have to show that 
\beq
\label{eq:restriction2}
\lim_{n\to \infty} \textVol_\g(U\setminus F_n^{-1}(U_n)) = 0 \textand 
\lim_{n\to \infty} \textVol_\h(V\setminus G_n^{-1}(U_n)) = 0.
\eeq
We prove the first equality; the second is proved by similar arguments. 

Fix $\e>0$. By Lemma~\ref{lm:A_n^e}, $ \textVol_\g(U\setminus A_n^\e)\to0$, hence it suffices to show that
\beq
\label{eq:showme1}
\textVol_\g((U\cap A_n^\e)\setminus F_n^{-1}(U_n))\to 0.
\eeq
Since $|dH_n|$ and $|dH_n^{-1}|$ are uniformly bounded in $n$ on $A_n^\e$, it follows from Hadamard's inequality (Lemma~\ref{lm:Hadamard}) that \eqref{eq:showme1} holds if and only if
\beq
\textVol_\h(H_n(U\cap A_n^\e)\setminus G_n^{-1}(U_n))\to 0.
\label{eq:showme2}
\eeq
As $G_n^{-1}(U_n) = (V\cap B_n)\cap H_n(U\cap A_n)$, 
\[
H_n(U\cap A_n^\e)\setminus G_n^{-1}(U_n) =
H_n(U\cap A_n^\e)\setminus (V\cap B_n) = H_n(U\cap A_n^\e)\setminus V,
\]
where in the last step we used the fact that $H_n(U\cap A_n^\e) \subset B_n$.
Let $y\in H_n(U\cap A_n^\e)\setminus V$. Using that fact that $H$ is an isometry,
\[
d_\N(y,H(H_n^{-1}(y))) = d_\M(H^{-1}(y),H_n^{-1}(y)) < \e_n,
\] 
where $\e_n = \sup_{y\in B_n} d_\M(H^{-1}(y),H_n^{-1}(y))$.
However, 
\[
H(H_n^{-1}(y)) \in H(U\cap A_n^\e) \subset H(U) = V,
\]
which implies that 
\[
d_\N(y,V) < \e_n.
\]
It follows that $H_n(U\cap A_n^\e) \setminus V$ is contained in an $\e_n$-tubular neighborhood of the boundary of $V$, hence
\[
\textVol_\h \brk{H_n(U\cap A_n^\e) \setminus V}  < C\e_n^{\dim \N} \to 0,
\]
where the constant $C$ depends of the length of the boundary of $V$. This completes the proof.
\end{proof}

\bigskip

We next show that Theorem~\ref{th:uniqueness} holds for the case of global frame fields, i.e., if $(\M_n,\g_n,\nabla_n)$ converges to two limits, then the uniform limit $H$ of the mappings $H_n$ is an isomorphism between the two limits. We do so in two steps: In Lemma~\ref{lm:uniqueness} we prove it under the additional assumption that it is the same sequence of frame fields $E_n$ that converges in the two limits. In Lemma~\ref{lm:uniqueness2} we relax this assumption.

\begin{lemma}
\label{lm:uniqueness}
Let $(\M_n,\g_n)$, $(\M,\g)$ and $(\N,\mathfrak{h})$ be compact Riemannian manifolds.
Let $E_n$ and $E^\M$ be frame fields on $\M_n$ and $\M$, respectively, and let $E^\N$ be a $\dim(\N)$-tuple of vector fields on $\N$.
Suppose that both 
\[
(\M_n,\g_n,E_n) \to (\M,\g,E^\M)
\Textand
(\M_n,\g_n,E_n) \to (\N,\h,E^\N)
\]
with respect to diffeomorphisms $F_n:A_n\subset\M \to \M_n$ and $G_n:B_n\subset\N\to \M_n$ (here, the pullbacks of the frame fields converge in $L^p$).
Then $H_\star E^\M = E^\N$, where $H:\M\to\N$ is the uniform limit of $H_n = G_n^{-1}\circ F_n$  
defined in Lemma~\ref{lm:metric uniqueness}. Furthermore, $E^\N$ is a frame field on $\N$.
\end{lemma}

\begin{proof}
We need to show that $H_\star E^\M - E^N = 0$. Since $H$ is the limit of $H_n$, we start by 
estimating $(H_n)_\star E^\M - E^\N$.
We fix some $\e>0$. Throughout this proof we will consider $H_n$ as a diffeomorphism $A_n^\e\to H_n(A_n^\e)$, where sets $A_n^\e$ are defined in Lemma~\ref{lm:A_n^e}.
By the standard inequality $|a+b|^p \le C(|a|^p + |b|^p)$ we get
\[
\begin{split}
\int_{H_n(A_n^\e)} |(H_n)_\star E^\M - E^\N|_\h^p \Volumeh 
&\le C\int_{H_n(A_n^\e)} |(H_n)_\star E^\M - G_n^\star E_n|_\h^p \Volumeh \\
&+ C \int_{H_n(A_n^\e)} |G_n^\star E_n - E^\N|_\h^p \Volumeh.
\end{split}
\]
The second addend tends to $0$ since $(\M_n,\g_n,E_n) \to (\N,\h,E^\N)$ with respect to the maps $G_n$. To show that the first addend tends to zero as well we observe that
\[
\begin{split}
\int_{H_n(A_n^\e)} |(H_n)_\star E^\M - G_n^\star E_n|_\h^p \Volumeh 
\le C \int_{A_n^\e} |E^\M - F_n^\star E_n|_\g^p \Volume \to 0, 
\end{split}
\]
by the uniform bound on $|dH_n|$ on $A_n^\e$ and Lemma~\ref{lm:Hadamard}.
We have thus shown that
\beq
\label{eq:tmp1}
\int_{H_n(A_n^\e)} |(H_n)_\star E^\M - E^\N|_\h^p \Volumeh  \to 0.
\eeq

The proof would be complete if we could replace $(H_n)_\star$ by $H_\star$ and $H_n(A_n^\e)$ by $\N$ in the limit $n\to\infty$. This is not yet possible since $H_n$ tends to $H$ on $A_n$ only uniformly, whereas the push-forward of frame fields with $H_n$ involves derivatives of $H_n$. 

Therefore, we will  show that $H_n\to H$ in $W^{1,p}$. Since Sobolev spaces are  easier to handle when the image is a vector bundle,
we fix an isometric immersion $\phi:(\N,\h)\to (\R^\nu,\euc)$ for  large enough $\nu$, where $\euc$ is the standard Euclidean metric. 
Since $H_n$ are uniformly Lipschitz on their restricted domains $A_n^\e$, the functions $\phi\circ H_n$ are $(1+3\e)$-Lipschitz mappings $A_n^\e\to\R^\nu$. 
By the McShane extension lemma \cite{Hei05}, there exists $L$-Lipschitz functions $\tH_n:\M\to \R^\nu$ (for some $L$ independent of $n$) that extend $\phi\circ H_n$ (the image of $\tH_n$ may no longer be a subset of the image of $\phi$). The functions $\tH_n$ converge to $\phi\circ H$ uniformly on $\M$, as 
\[
\begin{split}
d_{\R^\nu}(\tH_n(p), \phi\circ H(p)) &\le 
	d_{\R^\nu}(\tH_n(p),\tH_n(\psi_n(p))) + d_{\R^\nu}(\tH_n(\psi_n(p)),\phi\circ H(\psi_n(p))) \\
&\qquad+ d_{\R^\nu}(\phi\circ H(\psi_n(p)),\phi\circ H(p))\\
	&= d_{\R^\nu}(\tH_n(p),\tH_n(\psi_n(p))) + d_{\R^\nu}(\phi\circ H_n(\psi_n(p)),\phi\circ H(\psi_n(p))) \\
&\qquad+ d_{\R^\nu}(\phi\circ H(\psi_n(p)),\phi\circ H(p))\\
&\le d_{\R^\nu}(\tH_n(p),\tH_n(\psi_n(p))) + d_\N(H_n(\psi_n(p)),H(\psi_n(p))) \\
&\qquad+ d_\N(H(\psi_n(p)),H(p))\\
&=  d_{\R^\nu}(\tH_n(p),\tH_n(\psi_n(p))) + d_\N(H_n(\psi_n(p)),H(\psi_n(p))) \\
&\qquad+ d_\M(\psi_n(p),p) \\
	&\le L\cdot d_\M(p,\psi_n(p)) + d_\N(H_n(\psi_n(p)),H(\psi_n(p))) + d_\M(\psi_n(p),p) \\
	&\le (L+1)\,\sup_{\M} d(\cdot,\psi_n(\cdot)) + \sup_{A_n^\e} d_\N(H_n(\cdot),H(\cdot)) \to 0.
\end{split}
\]
Here $\psi_n$ is a mapping $\M\to A_n^\e$ satisfying
\[
\psi_n|_{A_n^\e} = \id
\Textand
\sup_{p\in\M} d_\M(p,\psi_n(p)) < \e_n
\]
for some $\e_n\to 0$; it is analogous to the mapping $\M\to A_n$ introduced in Lemma~\ref{lm:metric uniqueness}. 
Lemma~\ref{lm:A_n^e} implies that we can choose indeed such a sequence $\e_n\to 0$.
In the passage from the first to the second line we used the fact that $\tH_n$ coincides with $\phi\circ H_n$ on the image of $\psi_n$.  In the passage from the second to the third line we used the fact that $\phi$ is distance reducing. In the passage from the third to the fourth line we used the fact that $H$ is an isometry. The rest follows from the uniform Lipschitz bound on $\tH_n$ and the uniform convergence of $\psi_n$ to $\id_\M$, and  the uniform convergence of $H_n$ to $H$ on $A_n^\e$.

Changing variables $x\mapsto \phi(x)$, \eqref{eq:tmp1} takes the form
\[
\int_{\phi(H_n(A_n^\e))} |(\tH_n)_\star E^\M - \phi_\star E^\N|_\euc^p \Vol_{\Cy{\phi_\star \h}}  \to 0,
\]
where we used the fact that $\tH_n$ coincides with $\phi\circ H_n$ on $A_n^\e$. It follows that
\[
\int_{A_n^\e} |d\tH_n\circ E^\M - \tH_n^*\phi_\star E^\N|_{\euc}^p \Vol_{\g} \le
C\int_{A_n^\e} |d\tH_n\circ E^\M - \tH_n^*\phi_\star E^\N|_{\tH_n^*\euc}^p \Vol_{H_n^\Cy{\star}\h} \to 0.
\]
Since $\tH_n\to \phi\circ H$ uniformly and $E^\N$ is smooth, we can replace $\tH_n^*$ by $(\phi\circ H)^*$. Since $d\tH_n$ is uniformly bounded by the Lipschitz constant, and since $\textVol_\g (\M\setminus A_n^\e)\to 0$, the integral over $A_n^\e$ can be replaced by an integral over $\M$, yielding
\[
\int_{\M} |d\tH_n\circ E^\M - H^*(d\phi\circ E^\N)|_{\euc}^p \Vol_{\g} \to 0.
\]
It follows that $d\tH_n$ converges in $L^p(\M;T^*\M\otimes\R^\nu)$ to the map
\[
E^\M \mapsto H^*(d\phi\circ E^\N).
\]
Since, in addition, $\tH_n$ converges uniformly to $\phi\circ H$, it follows that $\tH_n$ converges to $\phi\circ H$ in $W^{1,p}(\M;\R^\nu)$, and in particular,
\[
d(\phi\circ H) \circ E^\M = H^*(d\phi\circ E^\N).
\]
Since $\phi$ is an embedding we can eliminate $d\phi$ on both sides, getting
\[
H_\star E^\M = E^\N.
\]

\end{proof}

The following lemma completes the proof of Theorem~\ref{th:uniqueness}:

\begin{lemma}
\label{lm:uniqueness2}
Let $(\M_n,\g_n,\nabla_n)$, $(\M,\g,\nabla^\M)$ and $(\N,\mathfrak{h},\nabla^\N)$ be compact Riemannian manifolds with metrically-consistent connections. Let $E_n$ and $D_n$ be $\nabla_n$-parallel frame fields on $\M_n$. Let $E^\M$ and $E^\N$ be $\nabla^\M$ and $\nabla^\N$-parallel frame fields on $\M$ and $\N$, respectively.
Suppose that 
\[
(\M_n,\g_n,E_n) \to (\M,\g,E^\M)
\Textand
(\M_n,\g_n,D_n) \to (\N,\mathfrak{h},E^\N),
\]
where the pullbacks of the frame fields converge in $L^p$.
Then there exists a matrix $Q\in \text{GL}_{\dim(\M)}(\R)$, such that $Q(H_\star E^\M) = E^\N$, where $H:\M\to\N$ is the Riemannian isometry defined in Lemma~\ref{lm:metric uniqueness}.
In particular, $H_* E^\M$ is a $\nabla^\N$-parallel frame field.
\end{lemma}

\begin{proof}
Given a Riemannian manifold $(X,\g)$, denote the subset of orthonormal frames of the frame bundle $\text{Fr}_p(TX)$ at a point $p$ by $O_p(X,\g)$.
Since $E^\M$ and $E^\N$ are parallel with respect to $\nabla^\M$ and $\nabla^\N$, which are $\g$- and $\h$- metrically-consistent connections, we can assume without loss of generality that $E^\M(p)\in O_p(\M,\g)$ and $E^\N(q)\in O_q(\N,\h)$ for every $p\in \M$ and $q\in \N$. If not, multiply $E^\M$ (and likewise $E^\N$) by a constant matrix $R$ such that $R E^\M$ is orthonormal.

$E_n$ and $D_n$ are both parallel with respect to the same connection $\nabla_n$, hence there exists a constant matrix $Q_n\in \text{GL}_{\dim(\M)}(\R)$ such that $Q_n E_n = D_n$. We now prove that the sequence $Q_n$ is bounded.

Fix some small $\e>0$, and denote
\[
R_n^\e = \BRK{x\in A_n: \dist\brk{F_n^\star E_n(x), O_x(\M,\g)} < \e }
\]
Since $F_n^\star E_n$ converges in $L^p$ to $E^\M\in O(\M,\g)$, it follows that 
\[
\lim_{n\to\infty} \textVol_\g \brk{\M\setminus R_n^\e}=0.
\]
Using Lemmas~\ref{lm:Hadamard}-\ref{lm:A_n^e},
\[
\begin{split}
\textVol_{\g_n} \brk{\M_n\setminus F_n(R_n^\e\cap A_n^\e)} 
	&= \textVol_{\g_n} \brk{\M_n\setminus F_n(A_n^\e)} + \textVol_{\g_n} \brk{F_n(A_n^\e) \setminus F_n(R_n^\e)} \\
	&\le \textVol_{\g_n} \brk{\M_n\setminus F_n(A_n^\e)} + (1+\e)^d \textVol_{\g} \brk{A_n^\e \setminus R_n^\e}\\
	& \underset{n\to\infty}{\longrightarrow} 0,
\end{split}
\]

Similarly, denoting
\[
S_n^\e = \BRK{y\in B_n: \dist\brk{G_n^\star D_n(y), O_y(\N,\h)} < \e },
\]
we obtain
\[
\lim_{n\to\infty} \textVol_{\g_n} \brk{\M_n\setminus G_n(S_n^\e)\cap F_n(A_n^\e)} = \lim_{n\to\infty} \textVol_{\g_n} \brk{\M_n\setminus G_n(S_n^\e\cap H_n(A_n^\e))}=0.
\]
In particular, for $n$ large enough, the set
\[
F_n(A_n^\e)\cap F_n(R_n^\e) \cap G_n(S_n^\e)
\]
is non-empty. For every point $x_n$ in it, we have the following:
\begin{enumerate}
\item $F_n^\star E_n(F_n^{-1} (x_n))$ is in an $\e$-neighborhood of the orthonormal frames $O_{F_n^{-1}(x_n)}(\M,\g)$.
\item Since $F_n^{-1}(x_n)\in A_n^\e$, $d_{F_n^{-1}(x_n)}F_n$ is in an $\e$-neighborhood of $\SO{\g,\g_n}$. 
\item  $G_n^\star D_n(G_n^{-1} (x_n))$ is in an $\e$-neighborhood of the orthonormal frames $O_{G_n^{-1}(x_n)}(\N,\h)$.
\item $d_{G_n^{-1}(x_n)}G_n$ is in an $\e$-neighborhood of $\SO{\h,\g_n}$.
\end{enumerate}
Therefore, both $E_n(x_n)$ and $D_n(x_n)$ are in some $O(\e)$-neighborhood of $O_{x_n}(\M_n,\g_n)$, where $O(\e)$ is independent of $n$. It follows that $Q_n$ is in $O(\e)$-neighborhood of $\SO{d}$, and in particular, it is uniformly bounded. 

It follows that there exists a converging subsequence (not relabeled) $Q_n\to Q$, and
\[
\begin{split}
\int_{A_n} |F_n^\star D_n-QE|^p_\g \Volume 
&= \int_{A_n} |Q_n (F_n^\star E_n) -QE^\M|^p_\g \Volume \\
&\hspace{-2cm}\le C\int_{A_n}  |Q (F_n^\star E_n-E^\M)|^p_\g +|(Q_n-Q) (F_n^\star E_n)|^p_\g \Volume\\
&\hspace{-2cm}\le C|Q|^p\int_{A_n} |F_n^\star E_n-E^\M|^p_\g \Volume + C|Q_n-Q|^p \int_{A_n} |F_n^\star E_n|^p_{\g} \Volume \\
&\hspace{-2cm}\le C\brk{\int_{A_n} |F_n^\star E_n-E^\M|^p_\g \Volume + |Q_n-Q|^p}\to 0,
\end{split}
\]
where we used the uniform boundedness of $F_n^\star E_n$.

If follows that 
\[
 (\M_n,\g_n,D_n) \to (\M,\g,Q E^\M) \textand (\M_n,\g_n,D_n) \to (\N,\mathfrak{h},E^\N),
\]
By Lemma~\ref{lm:uniqueness}, $Q (H_\star E^\M) = H_\star (QE^\M) = E^\N$. 
In particular, since $E^\N$ is a frame field, $Q$ is not singular, and the proof is complete.
\end{proof}
\bigskip

\section{Discussion}
\label{sec:discussion}

In this paper we prove that the limit of a specific sequence of manifolds with an increasing number of edge-dislocations is a smooth flat manifold endowed with a metrically-compatible non-symmetric flat connection, i.e. a Weitzenb\"ock manifold. 
Both the limit manifold and the limit connection are defined uniquely by the parameters $a$,$b$ and $\e$.
In particular, the limit remains unchanged if the dislocations in the sequence $\M_n$ are not located at the centers of each building block.

Moreover, the dislocation magnitude $\e$ of each block is determined by two parameters: $\e=2d\sin \theta$, where $\theta$ is the disclination angle and $d$ is the length of the dislocation line.
The metric limit is indifferent to the values of $d$ and $\theta$ as long as $\e$ remains fixed, and these values may change from one $\M_n(i,j)$-block to another.

For the connection limit to  hold, there is an additional constraint: -- the lengths $d_n$ of the dislocation lines must tend to zero faster than $n^{-1}$ (in our construction $d_n=O(n^{-2})$ since $\theta$ is fixed and $\e_n=O(n^{-2})$). If the length of the dislocation line is comparable to the cell size, the removal of the dislocation lines from $\tM_n$ changes the distance function significantly (see the proof of  Corollary~\ref{cy:conv}). This observation is consistent with the fact that the notion of ``curvature dipole" is ambiguous when it is not clear to which dipole each monopole (singularity) belongs.

To conclude, the limit is determined by the orientation of the dislocation lines and the magnitude of the dislocations -- that is to say, by the parallel vector fields, which are the Burgers vector fields of individual $\M_n(i,j)$-blocks. In our case the Burgers vector fields are equal to  $(\e/n^2)\,\pl_y$ (they can be calculated from the monodromy, see \cite{KMS14}). As a result, the total Burgers vector associated with a loop encircling $\alpha\beta n^2$ dislocations, $0<\alpha,\beta<1$, is $\alpha \beta\e\, \partial_y$.

The torsion field of the limit connection $\nabla$ is given by
\[
T = \frac{1}{r} \,dr\wedge d\vp\otimes {\pl_\vp}.
\]
It is the density of the Burgers field in the following sense: 
let $\Pi^p$ be the parallel transport operator to an arbitrary reference point $p$. Let 
\[
D = [r_1,r_1 + \alpha b]\times[\vp_1, \vp_1 + \beta \e/b] \subset\N
\]
a domain whose boundaries are $\nabla$-geodesics.
Using the fact that $r^{-1}\partial_\theta$ is a $\nabla$-parallel vector field,
\[
\int_D \Pi^p T = \alpha\beta\e \,(r^{-1}\partial_\theta)_p,
\]
which is the image under $F_n$ of $\alpha\beta\e\,\partial_y$ at $F_n(p)$.

Every metrically compatible connection of a two-dimensional manifold can be written as
\[
\nabla_X Y = \nabla_X^{LC} Y + \g(X,Y)V - \g(V,Y)X
\]
for some vector field $V$, where $\g$ is the metric and $\nabla^{LC}$ is the Levi-Civita connection  (see \cite{AT04} for details).
In our case, a simple calculation shows that this vector field  is $V=r^{-1}\pl_r$. $V$ can be interpreted, in a sense, as the continuum limit of the dislocation lines in $\M_n$.

While in our example, both the connections $\nabla_n$ on $\M_n$ and $\nabla$ on $\M$ admit global parallel frame fields, the notion of convergence given in Definition~\ref{df:convergence} relies only on the existence of local parallel frame fields. This gives some flexibility to include convergence to manifolds endowed with connections that are only locally-flat, for example,  edge-dislocations on a cone. 

We conclude this paper by raising \Cy{several} natural questions, which will be dealt in subsequent publications:

\begin{enumerate}
\item The example presented in this paper is a very specific one, with all the dislocations aligned in the same direction, resulting in a fairly simple limit torsion field. What other torsion fields can be obtained as limits of edge-dislocations in the sense of Definition~\ref{df:convergence}? For example, which  simply connected Weitzenb\"ock manifolds can be obtained as a limit of locally-flat Riemannian manifolds, each endowed with its Levi-Civita connection?

\Cy{
\item Another natural extension of this work is to account for continuous-distributed screw-dislocations, or more generally, distributed dislocations of both types. Screw-dislocations differ from edge-dislocations in that they are inherently three-dimensional.
The notion of convergence developed in this paper is independent of dimension and is therefore expected to apply in the more general case. On the other hand, the construction of a manifold with singular defects presented in this paper is two-dimensional, and as such cannot generate screw-dislocations.
}

\item In what way does the limit connection (or equivalently, the torsion field) manifest in the mechanical or elastic properties of the manifold? Assuming that the manifolds $\M_n$ represent elastic bodies with some elastic energy density, what is the limit elastic energy density on the limit manifold $\N$? This relates to a general question of $\Gamma$-convergence of elastic energy functionals in a limit of converging metrics.
\end{enumerate}

\paragraph{Acknowledgements}
We are very grateful to Jake Solomon for many fruitful discussions and enlightening ideas throughout  the evolution of this paper.
We are also grateful to Marcelo Epstein for suggesting us this homogenization problem, and to Pavel Giterman for his valuable comments.
This work was partially funded by the Israel Science Foundation and by the Israel-US Binational Foundation.

\appendix
\section{Proof of Lemma~\ref{lm:2}}

In this appendix we prove the Lemma:

\begin{quote}
{\itshape
Let $a,b,\e>0$ and $\theta\in(0,\pi/2)$ be given. 
Let $T_{n,i,j}$ be the natural intrinsic distance preserving mapping,
\[
T_{n,i,j} : \partial \N_n(i,j) \to \partial \tM_n(i,j).
\] 
Then, there exists a constant $c>0$ independent of $n,i,j$, 
such that
\[
\max_{x,y\in \partial\N_n(i,j)} |d(x,y) - d_n(T_{n,i,j}(x),T_{n,i,j}(y))| < \frac{c}{n^2}.
\]
}
\end{quote}

Recall that
\[
\tM_n(i,j) = \tR(a_{n,i},b_i,\theta,\e_n)
\Textand
\N_n(i,j) = \N(a_{n,i},b_i,\e_n),
\]
where $a_{n,i},b_i = O(1/n)$ and $\e_n = O(1/n^2)$. Here 
$\tR(\alpha,\beta,\theta,\delta)$ is the building block of  our locally-flat manifolds with defects, whereas 
$\N(\alpha,\beta,\delta)$ is a the sector of angle $\delta/\beta$ of an annulus of inner radius $\alpha\beta/\delta$ and outer radius $\alpha\beta/\delta + \beta$.

To prove the lemma, it is sufficient to prove that for
\[
\frac{c}{n} < \alpha, \beta < \frac{C}{n} \textand \delta<\frac{C'}{n^2},
\]
where $c,C,C'$ are  positive constants, the natural intrinsic-distance preserving map,
\[
T: \pl \N(\alpha,\beta,\delta) \to \pl \tR(\alpha,\beta,\theta,\delta),
\]
satisfies
\begin{align}
\label{eq:disT}
\max_{x,y\in \partial\N(\alpha,\beta,\delta)} |d_\N(x,y) - d_{\tR}(T(x),T(y))| < \frac{\tilde{C}}{n^2},
\end{align}
for  $\tilde{C}$ that depends only on $c,C$ and $C'$. Here $d_{\N}$ and $d_{\tR}$ are the respective  distance functions in $\N(\alpha,\beta,\delta)$ and $\tR(\alpha,\beta,\theta,\delta)$.

The proof is based on showing that for large $n$ both $\N(\alpha,\beta,\delta)$ and $\tR(\alpha,\beta,\theta,\delta)$ are almost isometric to a Euclidean rectangle,  $R(\alpha,\beta)$, with edges of length $\alpha,\beta$.
We construct two mappings, $S: \tR(\alpha,\beta,\theta,\delta)\to R(\alpha,\beta)$ and $S':R(\alpha,\beta)\to\N(\alpha,\beta,\delta)$, such that $T^{-1}:\pl\tR(\alpha,\beta,\theta,\delta)\to\pl\N(\alpha,\beta,\delta)$ is the restriction of $S'\circ S$ to the boundary. We then show that the distortions of both $S$ and $S'$ are $O(n^{-2})$, hence so is the distortion of their composition.

\paragraph{Construction of $S'$:} We endow both $\N(\alpha,\beta,\delta)$ and $R(\alpha,\beta)$ with Euclidean coordinates,
\[
\N(\alpha,\beta,\delta) = \BRK{(r\cos t,r\sin t) : (r,t)\in \Brk{\frac{\alpha\beta}{\delta},\frac{\alpha\beta}{\delta}+\beta}\times\Brk{0,\frac{\delta}{\beta}}},
\]
\[
R(\alpha,\beta) = \BRK{ (x,y)\in \Brk{\frac{\alpha\beta}{\delta},\frac{\alpha\beta}{\delta}+\beta}\times\Brk{0,\alpha}},
\]
and define $S'$ by
\[
S'(x,y) := \brk{ x\cos \brk{\frac{y\delta}{\alpha\beta}}, x\sin \brk{\frac{y\delta}{\alpha\beta}}}.
\]
This mapping is bijective. For all $(x,y)\in R(\alpha,\beta)$,
\[
| S'(x,y) - (x,y) |^2 = x^2\brk{1-\cos \brk{\frac{y\delta}{\alpha\beta}}}^2 + \brk{x\sin \brk{\frac{y\delta}{\alpha\beta}} - y}^2 = O(n^{-2}),	
\]
where we used the fact that $x = O(1)$, $\alpha,\beta, y = O(n^{-1})$ and $\delta = O(n^{-2})$.
It follows that for every two points $p,q\in R(\alpha,\beta)$,
\begin{align}
\label{eq:disS'}
|d_R(p,q) - d_{\R^2}(S(p),S(q))| < \frac{\tilde{C}}{n^2}.
\end{align}
Observe that since $\N(\alpha,\beta,\delta)$ is not convex $d_{\N}$ is not just a restriction of the Euclidean distance $d_{\R^2}$. 
The distance $d_\N$  between any two points in $\N(\alpha,\beta,\delta)$ cannot, however, be larger than the Euclidean distance in $\R^2$ plus $O\brk{\frac{\alpha\delta^2}{\beta^2}}=O(n^{-3})$. Hence the estimate~\eqref{eq:disS'} holds also with $d_{\N}$ replaced by $d_{\R^2}$.

\paragraph{Construction of $S$:}
$\tR(\alpha,\beta,\theta,\delta)$ can be constructed by gluing two euclidean hexagons.
We  define a bijective map $S:\tR(\alpha,\beta,\theta,\delta)\to R(\alpha,\beta)$ by defining it in an appropriate way on each hexagon.

Let $\tR_I$ be one of the hexagons, 
with the following Euclidean coordinates:
\[
\tR_I = \BRK{(x,y) : x\in[0,\beta], y\in
							\begin{cases}
							[0,a_1] 			& x\in[0,b_1]\\
							[0,a_1+\tan(x-b_1)]	& x\in(b_1,b_1+b_2\cos\vp]\\
							[0,a_1+\delta/2]		& x\in(b_1+b_2\cos\vp,\beta]
							\end{cases}
		}
\]
where $b_2=\frac{\delta}{2\sin\vp}$ is the distance between the singular points, $\beta=b_1+b_2\cos\vp+b_3$, and $a_1<a$ (the respective length in the other hexagon is $a-a_1$).
Denote
\[
a(x) = \sup \{ y: (x,y)\in \tR_\Raz{I}\}.
\]
Now define a bijective mapping $S'_I:\tR_I\to R(a_1,\beta)$ by
\[
S'_I((x,y)) = \brk{x,\frac{a_1}{a(x)}y}
\]
(we use the fact that $\theta\le\pi/2$, otherwise we construct a slightly different coordinate system).
A similar construction is used to define $S'_{II}: \tR_{II} \to R(\alpha-a_1,\beta)$. Gluing both maps together we get a bijective $S': \tR(\alpha,\beta,\theta,\delta)\to R(\alpha,\beta)$.
Now,
\[
\begin{split}
|S'_I((x,y))-(x,y)| &= \left| \frac{a_1}{a(x)}y-y \right| \le \brk{ a_1+\frac{\delta}{2}} \brk{ 1- \frac{a_1}{a_1+\frac{\delta}{2}}} = O(n^{-2}),
\end{split}
\]
and similarly for $S'_{II}$.
Like with $\N(\alpha,\beta,\delta)$, the hexagons are not convex, but it can easily be seen that 
\[
|d_{\tR_I}(x,y)-d_{\R^2}|\le\delta/2=O\brk{n^{-2}},
\]
hence for every two points $P,Q\in \tR_I$,
\begin{align}
\label{eq:disS}
|d_{\tR_I}(P,Q) - d_R(S'_I(P),S'_I(Q))| < \frac{\tilde{C}}{n^2},
\end{align}
and similarly for the second hexagon.

By construction $T^{-1}$ is a composition of the restriction of $S$ and $S'$ to the boundaries, hence by \eqref{eq:disS'},\eqref{eq:disS} we obtain \eqref{eq:disT}, which completes the proof.
$\blacksquare$

\section{Proof of Lemma~\ref{lm:3}}

In this appendix we prove the Lemma:

\begin{quote}
{\itshape
For every  $n\in\bbN$ and $p,q\in Y_n$, the shortest path in $\N$ connecting $p$ and $q$ intersects at most $3n$ out of the $n^2$ sectors $\N_n(i,j)$. 
Likewise, for every  $n\in\bbN$ and $p,q\in X_n$, the shortest path in $\tM_n$ (viewed as a metric space) connecting $p$ and $q$ intersects at most $3n$ out of the $n^2$ ``rectangles" $\tM_n(i,j)$. }
\end{quote}

Let $p,q\in Y_n$ and let $\gamma$ be the shortest path in $\N$ between them.
Assume that $\gamma$ intersects $k$ sectors $\N_n(i,j)$, and denote their indices by
\[
(i_1,j_1),\ldots,(i_k,j_k),
\]
where $p\in \N_n(i_1,j_1)$ and $q\in \N_n(i_k,j_k)$.

We prove that $k\le 3n$ by observing that $j_r-j_{r+1}$ never changes sign (in the weak sense, it may be $0$), and $i_r - i_{r+1}$ does not change sign more than once, which immediately implies $k\le 3n$.
This follows from the fact that the shortest path between a point in $\N_n(i,j)$ and a point in $\N_n(i',j)$ only passes through sectors $\N_n(\cdot, j)$, and a shortest path between a point in $\N_n(i,j)$ and a point in $\N_n(i,j')$ only passes through sectors $\N_n(i',\cdot)$ with $i'\le i$.
The same reasoning holds also for $\tM_n$, with its building blocks $\tM_n(i,j)$.
$\blacksquare$


\bibliographystyle{amsalpha}

\providecommand{\href}[2]{#2}
\providecommand{\arxiv}[1]{\href{http://arxiv.org/abs/#1}{arXiv:#1}}
\providecommand{\url}[1]{\texttt{#1}}
\providecommand{\urlprefix}{URL }

\end{document}